\documentclass{tac}

\usepackage{amssymb}
\usepackage{amsmath}
\usepackage{xfrac}
\usepackage{amscd}
\DeclareMathAlphabet{\mathpzc}{OT1}{pzc}{m}{it}
\usepackage[all, 2cell]{xy}

\renewcommand{\sfdefault}{iwona}
\DeclareMathAlphabet{\mathbfsf}{\encodingdefault}{\sfdefault}{bx}{n}

\newcommand{\define}[1]{\textbf{#1}}
\newcommand{\cosheaf}[1] {{\mathsf{{#1}}}}

\newcommand{\Rspace}        	{{\mathbb R}}

\newcommand{\Ispace}        	{{\mathbb I}}

\newcommand{\RR}{\mathbb{R}}
\newcommand{\CC}{\mathbb{C}}
\newcommand{\ZZ}{\mathbb{Z}}

\newcommand{\R}{{\mathbb{R}}}

\newcommand{\Z}{{\mathbb{Z}}}

\renewcommand{\emptyset}{\varnothing}
\newcommand{\Sstrat}		{\mathpzc{S}}
\newcommand{\Tstrat}		{\mathpzc{T}}

\newcommand{\Ccat}          	{{\mathcal{C}}}

\newcommand{\Ocat}          	{{\mathcal{O}}}

\newcommand{\Ucat}          	{{\mathcal{U}}}

\newcommand{\Open}          	{{\mathsf{Open}}}
\newcommand{\Set}          	{{\mathsf{Set}}}

\newcommand{\Vect}          	{{\mathsf{Vec}}}

\newcommand{\Strat}          	{{\mathsf{Strat}}}

\newcommand{\Cov}          	{{\mathsf{Cov}}}
\newcommand{\Csh}          	{{\mathsf{Csh}}}
\newcommand{\Ent}          	{{\mathsf{Ent}}}

\newcommand{\Basic}          	{{\mathsf{Basic}}}
\newcommand{\Cat}          	{{\mathsf{Cat}}}

\newcommand{\Loc}          	{{\mathsf{Loc}}}

\newcommand{\Afunc}          	{{\mathsf{A}}}
\newcommand{\Bfunc}          	{{\mathsf{B}}}

\newcommand{\Ffunc}          	{{\mathsf{F}}}
\newcommand{\Gfunc}          	{{\mathsf{G}}}

\newcommand{\Ifunc}          	{{\mathsf{I}}}
\newcommand{\Jfunc}          	{{\mathsf{J}}}

\newcommand{\Lfunc}          	{{\mathsf{L}}}

\newcommand{\Qfunc}          	{{\mathsf{Q}}}
\newcommand{\Rfunc}          	{{\mathsf{R}}}

\newcommand{\Iffunc}          	{{\mathtt{I}}}
\newcommand{\Jffunc}          	{{\mathtt{J}}}

\newcommand{\Hom}          	{{\mathsf{Hom}}}
\newcommand{\colim}		{\mathsf{colim}\;}

\newcommand{\id}			{\mathrm{id}}

\makeatletter
\providecommand{\leftsquigarrow}{%
  \mathrel{\mathpalette\reflect@squig\relax}%
}
\newcommand{\reflect@squig}[2]{%
  \reflectbox{$\m@th#1\rightsquigarrow$}%
}
\makeatother

\newcommand{\calR}{{\mathcal R}}

\newcommand{\cU}{{\mathcal U}}

\usepackage{tikz-cd}

\input diagxy

\usepackage[colorlinks=true]{hyperref}
\hypersetup{allcolors=[rgb]{0.1,0.1,0.4}}

\author{Justin Curry and Amit Patel}

\thanks{
The first author was supported by the National Science Foundation 
(previously NSF-RTG DMS-10-45153 and currently CCF-1850052) 
and the Air Force Office of Scientific Research (FA9550-12-1-0136).
The second author was supported by the 
National Science Foundation (DMS-1128155 and CCF-1717159).
}

\address{Department of Mathematics and Statistics, University at Albany SUNY\\
Albany, New York, USA\\ [5pt]
Department of Mathematics, Colorado State University \\
Fort Collins, Colorado, USA \\
}

\title{Classification of Constructible Cosheaves}

\copyrightyear{2020}


\eaddress{jmcurry@albany.edu\CR akpatel@colostate.edu}

\newtheoremrm{defn}{Definition}[section]
\newtheoremrm{prop}[defn]{Proposition}
\newtheoremrm{lem}[defn]{Lemma}
\newtheoremrm{thm}[defn]{Theorem}
\newtheoremrm{cor}[defn]{Corollary}
\newtheoremrm{rmk}[defn]{Remark}
\newtheoremrm{ex}[defn]{Example}

\newtheoremrm{rem}{Remark}

\begin{document}
\maketitle
\begin{abstract}
In this paper we prove an equivalence theorem originally observed by Robert MacPherson. 
On one side of the equivalence is the category of cosheaves that are constructible with respect to a 
locally cone-like stratification. 
Our constructibility condition is new and only requires that certain inclusions of open sets are sent to isomorphisms.
On the other side of the equivalence is the category of functors from the entrance path category, 
which has points for objects and certain homotopy classes of paths for morphisms. 
When our constructible cosheaves are valued in $\Set$ we prove an additional equivalence with 
the category of stratified coverings. 
\end{abstract}

%%%%%%%%%%%%%%%%%%%%%%%%%%%%%%%%
%%%%%%%%%%%%%%%%%%%%%%%%%%%%%%%%
\section{Introduction}
%%%%%%%%%%%%%%%%%%%%%%%%%%%%%%%%
%%%%%%%%%%%%%%%%%%%%%%%%%%%%%%%%

To begin, we consider a motivating question.
\begin{center}
\begin{quote}
 ``Given a stratified map $f:Y\to X$, how are the path components of $f^{-1}(x)$ organized, over all $x\in X$?''
 \end{quote}
\end{center}
Informally speaking, a map is stratified if it is proper and there is a partition $\Sstrat$ of $X$ into manifolds so that over each manifold $S \in \Sstrat$ the map $f^{-1}(S)\to S$ is a fiber bundle.

One perspective on this question uses the open sets of $X$ to define a functor
$\cosheaf{F}:\Open(X)\to\Set$ that assigns to each open $U$ the set $\pi_0(f^{-1}(U))$ of path components.
This functor satisfies a gluing axiom reminiscent of the usual van Kampen theorem, which makes it the prototypical example of a cosheaf of sets.
The assumption that $f$ is stratified endows $\cosheaf{F}$ with a strong property called constructibility.

Robert MacPherson observed that paths in $X$ can be used to organize $\pi_0(f^{-1}(x))$.
A path from $x$ to $x'$ lying entirely in a single stratum $S \in \Sstrat$ induces an isomorphism $\pi_0(f^{-1}(x)) \to \pi_0(f^{-1}(x'))$.
Suppose $x'$ lies in a stratum on the frontier of $S$.
MacPherson noticed that a path from $x$ to $x'$ that leaves one stratum
only by entering into a lower-dimensional stratum defines a map $\pi_0(f^{-1}(x)) \to \pi_0(f^{-1}(x'))$.
Such paths are called entrance paths and this map between components is invariant under certain homotopic perturbations.
One then thinks of this data as a functor $\Ffunc: \Ent(X, \Sstrat) \to \Set$ from the entrance path category $\Ent(X, \Sstrat)$
whose objects are points of $X$ and whose morphisms are certain homotopy classes of entrance paths.

One of the results of this paper is a proof that these two perspectives are equivalent.
This equivalence is a consequence of a more general result, which we call the Classification Theorem for Constructible Cosheaves (Theorem~\ref{thm:classification}).
This theorem proves that constructible cosheaves valued in a bi-complete category $\Omega$ are equivalent to representations of the entrance path category valued in $\Omega$.
The purpose of this paper is to provide a concise and self-contained proof of this theorem.
This requires two novel definitions:
\begin{itemize}
\item
A constructible cosheaf over a stratified space $(X, \Sstrat)$ is usually defined as a functor
$\cosheaf{F} : \Open(X) \to \Omega$ that satisfies a gluing axiom and that restricts to a locally constant cosheaf over each stratum.
Our version of the constructibility condition (Definition \ref{defn:constr-cosheaf}) simply asks that for certain pairs of open sets $U \subseteq V$,
the associated map $\cosheaf{F}(U)\to\cosheaf{F}(V)$ is an isomorphism.
Our definition of a constructible cosheaf implies the usual definition.
\item
As originally described by MacPherson, two entrance paths should be regarded as equivalent if they are related by a homotopy through entrance paths.
Following an idea of Treumann~\cite{treumann}, we declare two entrance paths to be equivalent if they are related by a finite number of elementary operations (Definition~\ref{defn:equiv-paths}), which shortcuts an entrance path by skipping past at most one stratum.
This definition is necessary for proving our van Kampen theorem (Theorem~\ref{thm:vanKampen}) for the entrance path category.
This is in turn necessary for proving the cosheaf axiom in Theorem~\ref{thm:classification}.
\end{itemize}
In addition to our main theorem (Theorem~\ref{thm:classification}),
we also define (Definition~\ref{defn:stratified-covering}) and classify stratified coverings (Theorem~\ref{thm:stratified_covering}).

We now review prior related work.
Kashiwara classified sheaves that are constructible with respect to a triangulation in his work on the Riemann-Hilbert Correspondence~\cite{kashiwara1984riemann}.
Shepard's thesis~\cite{shepard} has a similar result to Kashiwara's, but in the setting of cell complexes.
Treumann explored a 2-categorical analog of our theorem in his thesis~\cite{treumann}.
Woolf~\cite{woolf} classified constructible sheaves and cosheaves valued in $\Set$ in terms of functors from the fundamental category of a homotopically stratified space with locally 0 and 1-connected strata.
His equivalence uses an intermediary equivalence with branched covers, which prevents that proof technique from generalizing to cosheaves valued in categories other than $\Set$.
Curry's thesis~\cite{curry} provides a classification result for constructible (co)sheaves of finite-dimensional vector spaces over spaces belonging to a geometric-analytic category.
Curry and Lipsky
developed an algorithmic proof of the Van Kampen theorem in the setting of homotopies with compatible smooth triangulations, but this proof spans over five pages~\cite[\S 11.2.2]{curry}.
Finally, we note that Lurie~\cite{lurie} and Ayala/Francis/Tanaka~\cite{AFT} provide similar equivalences in the $\infty$-category setting, but it is unclear if their results can be used to deduce our own.
Regardless, our definitions of the involved structures are unique and our treatment should be understandable by anyone with a basic knowledge of point-set topology and category theory.

%%%%%%%%%%%%%%%%%%%%%%%%%%%%%%%%
%%%%%%%%%%%%%%%%%%%%%%%%%%%%%%%%
\section{Stratified Spaces}
\label{sec:stratified_spaces}
%%%%%%%%%%%%%%%%%%%%%%%%%%%%%%%%
%%%%%%%%%%%%%%%%%%%%%%%%%%%%%%%%

The core philosophy of stratification theory is to develop a good theory of spaces that are built using manifold pieces.
However, making precise how a space is ``built out of manifolds'' requires considerable work and there is no unique way of doing so.
For example, we note that in~\cite[Rmk. 4.1.9]{schurmann2012topology} Sch\"{u}rmann reviews 14 different notions of a stratified space, many of which overlap or imply the other.
Generally speaking, spaces defined using polynomial inequalities are considered to be stratified spaces, whereas the Cantor set is not.

We have chosen to work with the class of conically stratified spaces.
The definition we use is similar to the definition of CS sets used by Siebenmann~\cite{Siebenmann1972}, but we drop the assumption that the link is compact.
Our definition then agrees with Lurie's definition of a conically stratified space~\cite[Def.~A.5.5]{lurie} when the totally ordered set $\{0,1,\ldots,n\}$ is used to encode the stratification, so our definition is a special case of his.
The starting point for this definition, like so many other definitions of a stratified space, is the notion of a filtration.

\begin{defn}[Filtered Spaces and Homeomorphisms]\label{defn:filtration}
An $n$-\define{step filtration} of a topological space $X$ is a nested sequence of \emph{closed} subspaces of the form
$$\emptyset = X^{-1} \subseteq \cdots \subseteq X^n = X.$$
A \define{filtered space} is a space equipped with an $n$-step filtration for some value of $n$.
We will often refer to $X^i$ is the $i^{th}$ \define{degree} or \define{step} in a filtration of $X$.
Additionally, we say the \define{formal dimension} of a filtered space is $\sup\; \{n \mid X^n \setminus X^{n-1} \neq \varnothing\}$.

If $X$ and $Y$ are two filtered spaces, then we say that a map $f:X \to Y$ is \define{filtration-preserving} if $f(X^i)\subseteq Y^i$ for every $i$. Such a map is a \define{filtration-preserving homeomorphism} if it restricts to a homeomorphism in each degree.
\end{defn}

\begin{ex}
$X=\RR^n$ has an $n$-step filtration where
$X^i=\emptyset$ for all $i<n$ and $X^n=\RR^n$.
For this filtration the formal dimension and the Hausdorff dimension of $\R^n$ agree.
Unless otherwise stated, this is the assumed filtration of $\RR^n$.
\end{ex}

\begin{defn}[Product of Filtered Spaces]\label{defn:product-filtration}
If $X$ has an $n$-step filtration and $Y$ has an $m$-step filtration,
then the \define{product filtration} of $X\times Y$ is an $(n+m)$-step filtration where
$(X\times Y)^k = \bigcup_{i+j=k} X^i \times Y^j.$
\end{defn}

The next ingredient in the definition of a conically stratified space is, unsurprisingly, the notion of a cone.
The following definition of the cone is non-standard, but we find that it offers certain advantages over the usual definition.
The authors learned of this definition from Lurie~\cite[Def.~A.5.3]{lurie}.

\begin{defn}[Cone]\label{defn:cone}
Given a topological space $L$, the \define{cone} on $L$ is the topological space
\[
C(L)= \big( L\times \RR_{>0} \big) \cup \{\star\},
\]
where $U \subseteq C(L)$ is open if and only if $U \cap \big( L\times \RR_{>0} \big)$ is open.
We further require that if $\star \in U$, then there is an $\epsilon>0$ such that $L\times (0,\epsilon)\subseteq U$.
\end{defn}

As the reader is probably aware, the usual definition of the cone starts with $L\times \RR_{\geq 0}$ and then collapses the subspace $L\times \{0\}$ to a point.
We compare and contrast these two definitions in the following remarks.

\begin{rmk}[Agreement when $L$ is compact]
Let $\sim$ be the equivalence relation on $L\times\R_{\geq 0}$ that identifies two distinct points $x \sim x'$ if $x,x'\in L\times\{0\}$, otherwise no two distinct points are identified.
Let $C_1(L)$ be the quotient space defined by this equivalence relation.
Let $C_2(L)$ be the cone on $L$ given by Definition~\ref{defn:cone}.
If we identify the equivalence class containing $L\times\{0\}$ with $\{\star\}$, then $C_1(L)$ and $C_2(L)$ are identical set-theoretically, but they have slightly different topologies as we will now see.

Notice there is a map $q: L\times \R_{\geq 0} \to C_1(L)$ that sends $L\times \{0\}$ to the distinguished cone point $\star$.
Let $V\subseteq C_1(L)$ be an open set in the traditional cone that contains the cone point.
This is open if and only if $q^{-1}(V)$ is open in $L\times \R_{\geq 0}$.
We know that $q^{-1}(V)$ contains a union of open sets of the form $W_i\times [0,\epsilon_i)$ for $\{W_i\}$ forming a cover of $L$.
If $L$ is compact then we can select a finite subcover of the $\{W_i\}$ and set $\epsilon$ to be the minimum of the $\epsilon_i$ used in the finite subcover.
This shows that when $L$ is compact the two notions of a cone agree.
However, if $L=\R$, for example, it is possible to have an open set in $C_1(L)$ that is not open in $C_2(L)$.
This implies that the identity map $C_1(L) \to C_2(L)$ is continuous, but not always a homeomorphism.
\end{rmk}

\begin{rmk}[Cone of the Empty Set]
Recall that the product of any space with the empty space is empty.
One of the advantages that Definition~\ref{defn:cone} offers is that the cone is defined via a \emph{union} with the one-point space $\{\star\}$, so even when $L=\emptyset$, we have that $C(L)=\{\star\}$.
In the usual definition described above, one has to introduce a special case that specifies that the cone of the empty set is the one-point set.
\end{rmk}

\begin{rmk}[Filtration of the Cone]
Let $L$ be a space with an $n$-step filtration.
Now consider $\RR_{>0}$ with the 1-step filtration, which is empty in degrees $-1$ and $0$, but $\RR_{>0}$ in degree 1.
The product filtration of $L\times \RR_{>0}$ is then simply $(L\times \RR_{>0})^{i+1}=L^i\times \RR_{>0}$.
Since $L^0$ is promoted to degree one in $L\times \RR_{>0}$, we simply place the cone point $\{\star\}$ in degree zero, which yields the $(n+1)$-step filtration of $C(L)$ where
\[
	C(L)^0 = \{\star\} \qquad \text{and} \qquad C(L)^{i+1} = \big( L^i \times \RR_{>0} \big) \cup \{\star\}.
\]
\end{rmk}

We now have all the ingredients necessary for defining a conically stratified space.

\begin{defn}[Conically Stratified Spaces]\label{defn:stratified-space}
An \define{$n$-dimensional conically stratified space} is a Hausdorff space $X$ equipped with an $n$-step filtration
\begin{equation*}
	\emptyset = X^{-1} \subseteq \cdots \subseteq X^n = X
\end{equation*}
where for each integer $d\geq 0$ and each point $x \in X^d - X^{d-1}$
there exists a distinguished open set $U_x \subseteq X$ containing $x$,
a filtered space $L_x$ of formal dimension $(n-d-1)$,
and a filtration-preserving homeomorphism
$$h_x:\Rspace^d \times C(L_x)\to U_x$$
taking $(0,\star)$ to $x$.
We call the space $L_x$ the \define{link} about $x$ and the open set $U_x$ a \define{basic open}.
A connected component of $X^d - X^{d-1}$ for any $d\geq 0$ is called a \define{d-dimensional stratum}.
We will denote the set of all strata by $\Sstrat$ and will refer to $X$ along with its stratification using the pair $(X,\Sstrat)$.
When there is no risk of confusion, we will call a conically stratified space simply a \define{stratified space}.
\end{defn}

\begin{rmk}
The class of conically stratified spaces includes the class of \define{CS sets} in the sense of Siebenmann~\cite{Siebenmann1972}, but it includes more spaces since we no longer assume the link is compact.
CS sets in turn include \define{topological stratified spaces}, as introduced by Goresky and MacPherson~\cite{GorMac83}.
Indeed, following the presentation of Friedman~\cite[Def. 2.3.10]{friedman}, topological stratified spaces can be viewed as \emph{recursive} CS sets, because the link $L_x$ is assumed to be a CS set as well.
\end{rmk}

\begin{ex}[A Non-Example]
Since the link of a point in a conically stratified space is only assumed to be filtered, one may wonder what happens when we let the link of a point be something pathological, like the Cantor set.
Let $L$ be the Cantor set, filtered as $L^{-1}=\varnothing$ and $L^0=L$.
Let $X=C(L)$ be the cone on the Cantor set.
At the cone point $\star$, $X$ has a conical neighborhood required by Definition~\ref{defn:stratified-space}.
However, for any point $x\neq \star$, there does \emph{not} exist a neighborhood $U_x$ homeomorphic to $\R$.
To see why, consider the projection $\pi(x)$ of $x$ to the link. Since every point of the Cantor set is a limit point any open set containing $\pi(x)$ will contain infinitely many other points in the Cantor set.
This proves the point $x$ in the product topology $L\times\R_{>0}$ will not have a neighborhood homeomorphic to $\R$.
\end{ex}

\begin{rmk}[Basic Opens]\label{rmk:basic-opens}
We call the distinguished open sets $U_x$ described in Definition~\ref{defn:stratified-space} \emph{basic} opens because they form a basis for the topology on $X$.
To verify this, note that for every $r>0$ and $s>0$, we can restrict the homeomorphism $h_x:\Rspace^d \times C(L_x)\to U_x$ to
those points $(v,\ell,t)\in \R^d \times C(L_x)$ with $||v||<r$ and $t < s$ to obtain another distinguished open $U_x(r,s)\subset X$.
This implies that the collection of distinguished opens, written $\Basic(X,\Sstrat)$, forms a basis for the topology on $X$.
Note that we are using the fact that points in the cone on the link $C(L_x)$ are equivalently indexed by points $(\ell,t)\in L_x \times \RR_{\geq 0}$.
\end{rmk}

\begin{defn}
Given a basic open $U_x \in \Basic(X,\Sstrat)$, its \define{associated stratum} is the stratum $S$ containing $x$.
We note that $S\in \Sstrat$ is the unique lowest dimensional stratum intersecting $U_x$.
We also call $U_x$ an \define{$S$-basic open} when it is convenient to do so.
We denote the collection of $S$-basic opens by $\Basic(X,S)$.
Clearly
\[
	\Basic(X,\Sstrat)= \bigcup_{S\in \Sstrat} \Basic(X,S).
\]
\end{defn}

In the following remark, we note how a basic open $U_x$ can also be regarded as a basic open for all nearby points $x'$ in the same associated stratum as $x$.
This is sometimes referred to as \emph{local triviality} of a stratification.
It further implies that the associated stratum to a basic open is the lowest dimensional stratum it intersects.

\begin{rmk}\label{rmk:change-of-cone-point}
Given a statum $S\subseteq X^d - X^{d-1}$, a point $x\in S$ and a basic open $U_x$ about $x$,
we can regard $U_x$ as a basic open neighborhood of every point $x' \in S \cap U_x \cong \RR^d$.
Indeed, for each such $x'$ the filtration-preserving homeomorphism $h_x: \R^d \times C(L_x) \to U_x$ carries some point $(v',\star)\in \R^d \times C(L_x)$ to $x'$.
By pre-composing $h_x$ with the map $A_{v'}\times \id$ where $A_{v'}(v)=v+v'$, one obtains a filtration-preserving homeomorphism $h_{x'}$ sending $(0,\star)$ to $x'$, thereby making $U_x$ a basic open for $x'$.
Consequently, we can drop the subscript and write $U \in \Basic(X, \Sstrat)$ when it is convenient.
\end{rmk}

One of the implications of Definition~\ref{defn:stratified-space} is that a basic open restricts to a coordinate chart for its associated stratum.

\begin{rmk}[Manifold Strata]\label{rmk:mfld-strata}
We say that a component of $X^d - X^{d-1}$ is a $d$-dimensional stratum because it is, in fact, a $d$-dimensional manifold.
To see this, consider a point $x\in X^d - X^{d-1}$ and a basic open $U_x$ containing $x$.
The filtration of
\[
\RR^d \times C(L_x) = \big( \RR^d \times L_x \times \RR_{>0} \big) \cup \big( \RR^d \times \{\star\} \big),
\]
combined with the assumption that $h_x$ is a filtered homeomorphism implies that that we have a homeomorphism
$$\RR^d \times \{\star\} \to U_x \cap X_d .$$
\end{rmk}

\begin{rmk}[Poset of Strata]
One of the consequences of the definition of a conically stratified space $(X,\Sstrat)$ is that the set of strata $\Sstrat$ satisfies the axiom of the frontier.
The axiom of the frontier says that if $S$ and $S'$ are two strata and if $S\cap \overline{S'}\neq \varnothing$, then $S\subseteq \overline{S'}$.
Lemma 2.3.7 of~\cite{friedman} proves that CS sets satsify the axiom of the frontier, but that proof goes through without the assumption that the link is compact.
Following Proposition 2.2.20 of~\cite{friedman}, we can define a partial order on $\Sstrat$ where $S \leq S'$ if and only if $S$ is contained in the closure of $S'$.
\end{rmk}

We conclude this section by noting how the collection of stratified spaces form a category.

\begin{defn}
A \define{stratum-preserving map} $f : (Y, \Tstrat) \to (X, \Sstrat)$ is a continuous map $f : Y \to X$ such that
for every $T\in \Tstrat$, there exists an $S \in \Sstrat$ where $f(T) \subseteq S$.
\end{defn}

\begin{ex}
Note that any open subset $U \subseteq X$ of a conically stratified space is also a conically stratified space.
If we write $(U, \Sstrat_U)$ for the the restriction of
$(X, \Sstrat)$ to $U$, then we have that the inclusion $\iota: (U, \Sstrat_U) \to (X, \Sstrat)$ is a stratum preserving map.
\end{ex}

\begin{defn}
Let \define{$\mathsf{Strat}$} be the category of conically stratified spaces and stratum-preserving maps.
\end{defn}

%%%%%%%%%%%%%%%%%%%%%%%%%%%%%%%%%%%%%%
\section{Constructible Cosheaves}
\label{sec:open_sets}
%%%%%%%%%%%%%%%%%%%%%%%%%%%%%%%%%%%%%%

In this section, we introduce the notion of a cosheaf and a {\em constructible} cosheaf.
The definition of a cosheaf by open sets only requires the structure of a topological space, whereas a constructible cosheaf requires the additional structure of a stratification.
Both of these notions dualize to the setting of sheaves and constructible sheaves.
It should be noted that the definition of a constructible cosheaf is not a simple dualization of the definition of a constructible sheaf that one might find in~\cite{KS}, for example.
The key difference is that we use the collection of basic cone-like open sets associated to a stratified space in an essential way.
This requires that we also consider a different notion of cover, as we explain below.

\begin{defn}
A \define{pre-cosheaf} $\cosheaf{F}$ assigns to each open set $U$ of a topological space $X$ an object $\cosheaf{F}(U)$ in some target category $\Omega$.
Moreover, whenever we have a pair of open sets $U\subseteq V$ a pre-cosheaf assigns a morphism $\cosheaf{F}(U\subseteq V): \cosheaf{F}(U) \to \cosheaf{F}(V)$ between the corresponding objects.
Said succinctly, a pre-cosheaf is a functor $\cosheaf{F}:\Open(X) \to \Omega$, where $\Open(X)$ denotes the poset of all open sets, partially ordered by inclusion.
\end{defn}

\begin{defn}
Let $x$ be a point of a topological space $X$ and let $\cosheaf{F}$ be a pre-cosheaf.
The \define{costalk} at $x$ of $\cosheaf{F}$, written $\cosheaf{F}(x)$, is the (inverse) limit of $\cosheaf{F}$ restricted to those open sets $U$ containing $x$.
Note that $\cosheaf{F}(x)$ is not $\cosheaf{F}$ applied to an open set.
However, for any open set $U$ containing $x$, we let
$\cosheaf{F}(x \in U) : \cosheaf{F}(x) \to \cosheaf{F}(U)$
denote the canonical morphism given by the definition of a limit.
\end{defn}

In order to make sure that costalks are well defined, we assume that $\Omega$ is a complete category.
Since the cosheaf axiom, described below, uses colimits indexed by a cover, we henceforth assume that $\Omega$ has all small limits and colimits, i.e.~we assume that $\Omega$ is a bi-complete category.

\begin{defn}\label{defn:cover}
A \define{cover} of an open set $U\subseteq X$ is a collection of open sets $\cU \subseteq \Open(X)$ so that the union $\cup_{U_i\in\cU} U_i =U$.
A \define{complete cover} of an open set $U$ is a cover $\cU$ of $U$ with the property that whenever $U_i,U_j\in \cU$, there exists a subcollection $\cU_{ij}\subseteq \cU$ that covers $U_i\cap U_j$.
\end{defn}

\begin{ex}[Complete Covers and Bases]
Complete covers provide a good notion of a cover when using a basis $\mathcal{B}$ of a topological space $X$.
If $U\in \Open(X)$ is an open set and $\mathcal{B}_U$ is the collection of basic opens inside of $U$, then $\mathcal{B}_U$ defines a complete cover of $U$.
In particular, for a conically stratified space $X$ the collection of all cone-like basic opens $\Basic(X, \Sstrat)$ is a complete cover of $X$.
\end{ex}

\begin{rmk}[Categorical Remark]
Any cover $\cU=\{U_i\}$ can be viewed as a poset where $U\leq U'$ whenever $U\subseteq U'$.
The inclusion of $\cU \subseteq \Open(X)$ then defines a map of posets, which is equivalent to saying that the map defines a functor between these two categories.
To say that $\cU$ covers $U$ is to say that the colimit of $\cU$, viewed as a diagram in $\Open(X)$, is $U$.
\end{rmk}

We now introduce our definition of a cosheaf.

\begin{defn}\label{defn:cosheaf}
Let $X$ be a topological space and let $\cosheaf{F}:\Open(X) \to \Omega$ be a pre-cosheaf.
Given a cover $\cU$ we use the notation $\cosheaf{F}|_{\cU}$ to refer to
the pre-composition of $\cosheaf{F}$ with the inclusion functor $\cU \subseteq \Open(X)$.
Note that $\cosheaf{F}|_{\cU}$ is a functor from $\cU \to \Omega$, i.e.~$\cosheaf{F}|_{\cU}$ is a diagram in $\Omega$ indexed by cover elements and inclusion of cover elements.
Let $\colim \cosheaf{F}|_{\cU}$ denote the colimit of this diagram, which is an object in $\Omega$. In other words
\[
  \colim \cosheaf{F}|_{\cU} := \varinjlim_{U_i\in \cU} \cosheaf{F}(U_i).
\]
We say that $\cosheaf{F}$ is a \define{cosheaf} if for every open set $U$ and every
\emph{complete cover} $\cU$ of $U$,
the universal morphism from the colimit $\colim \cosheaf{F}|_{\cU} \to \cosheaf{F}(U)$
is an isomorphism.
\end{defn}

\begin{rmk}[Different Cosheaf Axioms]
Traditionally, the sheaf or cosheaf axiom is phrased using \v Cech covers instead of complete covers.
A \define{\v{C}ech cover} of an open set $U$ is a cover $\cU$ of $U$ with the property that whenever two cover elements $U_1, U_2 \in \cU$
have a non-empty intersection, then $U_1 \cap U_2$ is also in $\cU$.
Note that every \v Cech cover is automatically a complete cover so our notion of a cosheaf includes the traditional definition of a cosheaf.
\end{rmk}

We now introduce the fundamental definition of this section.

\begin{defn}\label{defn:constr-cosheaf}
Let $(X,\Sstrat)$ be a conically stratified space.
A cosheaf
$\cosheaf{F} : \Open(X) \to \Omega$ is \define{$\Sstrat$-constructible} if
\begin{quote}
for each stratum $S\in \Sstrat$ and every pair of $S$-basic opens $U\subseteq V$ the morphism $\cosheaf{F}(U) \to \cosheaf{F}(V)$ is an isomorphism.
\end{quote}
A \define{morphism} $\phi : \cosheaf{F} \to \cosheaf{G}$ between two $\Sstrat$-constructible cosheaves is simply a natural transformation of functors.
Let $\Csh_{\Omega}(X, \Sstrat)$ denote the category whose objects are $\Sstrat$-constructible cosheaves valued in $\Omega$ and whose morphisms are
natural transformations.
\end{defn}

We now mention a particularly important special class of constructible cosheaves, where the assigned objects do not vary among strata.

\begin{defn}\label{defn:locally-constant}
Let $(X,\Sstrat)$ be a conically stratified space and
let $\cosheaf{F}:\Open(X)\to \Omega$ be a cosheaf.
We say that $\cosheaf{F}$ is \define{locally constant} if
\begin{quote}
for every pair of basic opens $U\subseteq V$, regardless of their associated strata, the morphism $\cosheaf{F}(U) \to \cosheaf{F}(V)$ is an isomorphism.
\end{quote}
The category of locally constant cosheaves $\Loc_{\Omega}(X, \Sstrat)$ is a full subcategory of $\Csh_{\Omega}(X, \Sstrat)$.
\end{defn}

\begin{rmk}
As observed in Remark~\ref{rmk:mfld-strata}, every basic open restricts to a basic open on each stratum.
Thus our definition implies that a constructible cosheaf restricts to a locally constant cosheaf over each stratum because every basic open is automatically associated to a single stratum.
\end{rmk}

We conclude this section with a remark.

\begin{rmk}
A morphism $\phi : \cosheaf{F} \to \cosheaf{G}$ of $\Sstrat$-constructible cosheaves
is locally constant over each stratum.
This is because for every $S\in \Sstrat$ and any pair of $S$-basic opens $U \subseteq V$ the following diagram commutes:
	\begin{equation*}
	\xymatrix{
	\cosheaf{F}(U) \ar[rr]^{\cosheaf{F}(U \subseteq V)}_\cong \ar[d]_{\phi(U)}
	&& \cosheaf{F}(V) \ar[d]^{\phi(V)} \\
	\cosheaf{G}(U) \ar[rr]_{\cosheaf{G}(U \subseteq V)}^\cong && \cosheaf{G}(V).
	}
	\end{equation*}
\end{rmk}

%%%%%%%%%%%%%%%%%%%%%%%%%%%%%%%%%%%%%%%%%%%%%%%%%
\section{The Entrance Path Category}
\label{sec:entrance_paths}
%%%%%%%%%%%%%%%%%%%%%%%%%%%%%%%%%%%%%%%%%%%%%%%%%

The fundamental groupoid of a space is a category whose objects are points
of the space and whose morphisms are homotopy classes of paths.
In this section, we introduce the entrance path category of a stratified space
which is similar to the fundamental groupoid with the exception that paths have to be directed in a precise sense: a path can leave one stratum only by entering a lower-dimensional one.
In the case that $(X,\Sstrat)$ has exactly one stratum for each component of $X$, the entrance path category is the fundamental groupoid of the space.
We conclude this section by defining representations of the entrance path category, which we ultimately prove are equivalent to constructible cosheaves.

\begin{defn}
Let $\Ispace$ be the closed interval $[0,1]$.
An \define{entrance path} is a stratification $(\Ispace, \Tstrat)$ of the interval and
a stratum preserving map $\alpha: ( \Ispace, \Tstrat ) \to (X, \Sstrat)$ such that the dimension of the stratum
containing $\alpha(t)$ is non-increasing with increasing $t$.
\end{defn}

By compactness of $\Ispace$, any stratification $(\Ispace, \Tstrat)$ has a finite number of $0$-strata.
It will be convenient to write $(\Ispace, \Tstrat)$ as $\big( \Ispace, \{0, \ldots, 1\} \big)$ where
$\{0, \ldots, 1\} \subset \Tstrat$ is the set of $0$-strata.
The concatenation of $\alpha$ followed by $\beta$, where $\beta(0)=\alpha(1)$,
is a new entrance path $\beta \ast \alpha$ defined as
	\begin{equation*}
	\beta \ast \alpha (t) =
		\begin{cases}
		\alpha(2t) & \text{for } 0 \leq t \leq \sfrac{1}{2} \\
		\beta(2t-1) & \text{for } \sfrac{1}{2} \leq t \leq 1.
		\end{cases}
	\end{equation*}

\begin{defn}
A \define{reparametrization} of an entrance path $\alpha : (\Ispace,\Tstrat) \to (X,\Sstrat)$ is an entrance path $\alpha \circ \phi : (\Ispace,\Tstrat') \to (X,\Sstrat)$
where $\phi : (\Ispace,\Tstrat') \to (\Ispace, \Tstrat)$ is an isomorphism in $\Strat$
that fixes endpoints, i.e. $\phi(0)=0$, $\phi(1)=1$.
Note that the number of strata in $\Tstrat$ and $\Tstrat'$ is the same.
\end{defn}

\begin{defn}
\label{defn:elementary_homotopy}
Let $\Delta$ be the $2$-simplex $[v_0v_1v_2]$ stratified by its vertices, edges, and face.
An \define{elementary homotopy} from an entrance path $\alpha : \big( \Ispace, \{0,\sfrac{1}{2},1\} \big) \to (X, \Sstrat)$ to an entrance path
$\beta : \big( \Ispace, \{0,1\} \big) \to (X, \Sstrat)$, written $\alpha \Rightarrow \beta$,
is a diagram in $\Strat$

	\begin{equation*}
	\label{eq:elementary_homotopy}
	\begin{gathered}
	\xymatrix{
	\big( \Ispace, \{0,\sfrac{1}{2},1\} \big) \ar[rdd]^{\tilde{\alpha}} \ar[rdddd]_\alpha & & \big( \Ispace, \{0,1\} \big) \ar[ldd]_{\tilde{\beta}} \ar[ldddd]^{\beta} \\
	&& \\
	& \Delta \ar[dd]^{h} & \\
	&& \\
	& (X, \Sstrat),
	}
	\end{gathered}
	\end{equation*}
satisfying the following conditions:
	\begin{itemize}
	\item $h|_{[v_0,v_1] \cup [v_1,v_2]} \circ \tilde{\alpha} =\alpha$ and $h|_{[v_0,v_2]} \circ \tilde{\beta} =\beta$;
	\item $h|_{[v_0,v_1,v_2] \setminus [v_1,v_2]}$ is mapped to a single stratum; and
	\item there is a basic open neighborhood $U \in \Basic(X, \Sstrat)$
	of the point $h(v_2) = \alpha(1) = \beta(1)$ such that the image of $h$ belongs to $U$.
	\end{itemize}
\end{defn}

If $\alpha \Rightarrow \beta$, then $\alpha(0) = \beta(0)$ and $\alpha(1) = \beta(1)$.
The image of the elementary homotopy $h$ intersects at most three strata in $\Sstrat$ and so does the entrance path $\alpha$.
The entrance path $\beta$ intersects at most two strata in $\Sstrat$.
We think of an elementary homotopy $\alpha \Rightarrow \beta$
as a shortcutting operation.
That is, $\beta$ may skip one of the strata $\alpha$ goes through.

\begin{defn}
We say two entrance paths $\alpha$ and $\beta$ are \define{elementarily related}, written $\alpha \rightleftharpoons \beta$,
if there are reparametrizations $\alpha \circ \phi_{\alpha}$ and $\beta \circ \phi_{\beta}$ and compositions
$\alpha \circ \phi_{\alpha} = \alpha_3 \ast \alpha_2 \ast \alpha_1$ and $\beta \circ \phi_{\beta} = \beta_3 \ast \beta_2 \ast \beta_1$
such that $\alpha_1 = \beta_1$, $\alpha_2 \Rightarrow \beta_2$
or $\beta_2 \Rightarrow \alpha_2$, and $\alpha_3 = \beta_3$.
\end{defn}

\begin{figure}[h]
\centering
\includegraphics[width=.9 \textwidth]{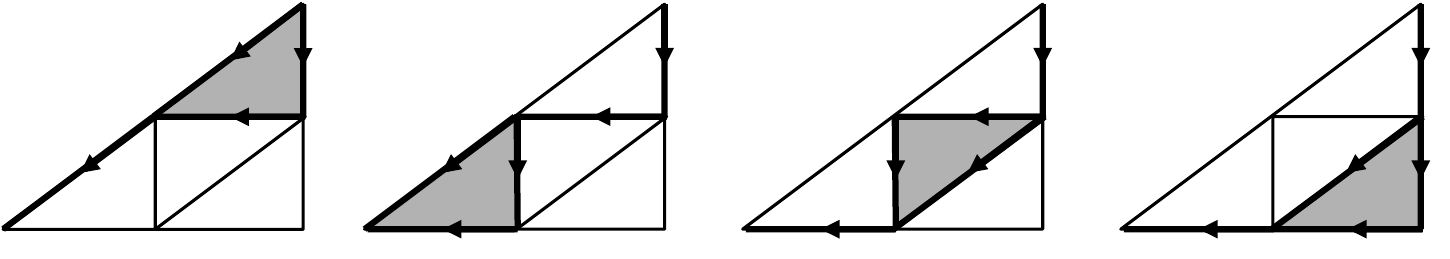}
\caption{An equivalence of entrance paths is a finite number of elementary relations.}
\label{fig:refined_homotopy}
\end{figure}

\begin{defn}
\label{defn:equiv-paths}
We say two entrance paths $\alpha$ and $\beta$ are \define{equivalent}, written $\alpha \approx \beta$, if there a finite number of elementary relations
$$\alpha \rightleftharpoons \gamma_1 \rightleftharpoons \gamma_2 \rightleftharpoons \cdots \rightleftharpoons \gamma_n \rightleftharpoons \beta.$$
See Figure \ref{fig:refined_homotopy}.
Note $\approx$ is an equivalence relation on the set of all entrance paths.
\end{defn}

\begin{ex}[Changing the Stratification of the Interval]
\label{ex:change_stratification}
Note that the data of an entrance path includes information about the stratification of the interval, but under certain situations this can be ignored.
For example, suppose $(X,\Sstrat)$ is a stratified space with only one stratum $S$ and $\gamma :(\Ispace, \Tstrat) \to (X,\Sstrat)$ and $\gamma' : (\Ispace, \Tstrat') \to (X,\Sstrat)$ are two entrance paths mapping into $S$ that are identical as paths, i.e.~$\gamma(t)=\gamma'(t)$ for all $t\in \Ispace$, but $\Tstrat$ and $\Tstrat'$ have a different number of strata in $\Ispace$.
We claim that $\gamma$ and $\gamma'$ are equivalent under our definition.

To see this, consider the simplest situation where we have identical paths $\gamma$ and $\gamma'$ where $\gamma$ has three zero-dimensional strata $\{0,t_1,1\}$ in $\Tstrat$ and $\gamma'$ has only the two necessary zero-dimensional strata $\{0,1\}$ in $\Tstrat'$.
Again, we assume that $\gamma(t)=\gamma'(t)$ for all $t\in \Ispace$.
Applying a reparameterization, we can replace $\gamma$ with $\alpha :(\Ispace,\{0,1/2,1\}) \to (X,\Sstrat)$ and $\gamma'$ with $\beta:(\Ispace, \{0,1\}) \to (X,\Sstrat)$ so that we still have that $\alpha(t)=\beta(t)$ for all $t\in \Ispace$.
Using barycentric coordinates we can give an explicit elementary homotopy $h$ relating $\alpha$ and $\beta$ by using the formula $h(t_0,t_1,t_2)=\alpha(\frac{1}{2}t_1 + t_2)=\beta(\frac{1}{2}t_1 + t_2)$.
In this setting, an elementary homotopy allows us to ``delete'' one stratum in the domain of $\alpha$ without consequence.

In the general setting, where $\gamma : (\Ispace,\Tstrat) \to (X,\Sstrat)$ is an entrance path that is contained in a single stratum $S$, we can delete zero-dimensional strata in $\Tstrat$ one-by-one to replace $\gamma$ with an equivalent entrance path $\gamma':(\Ispace, \{0,1\}) \to (X,\Sstrat)$.
\end{ex}

\begin{defn}\label{defn:ent-cat}
The \define{entrance path category} $\Ent(X, \Sstrat)$ has an object for each point $x \in X$.
A morphism from $x\in X$ to $y\in X$ is an equivalence class of entrance paths starting at $x$ and ending at $y$.
\end{defn}

\begin{ex}\label{ex:R}
Suppose the real line $\R$ is stratified into $n$ vertices $x_1 < \cdots < x_n$ and $n+1$ open intervals with $e_0 = (-\infty, x_1)$, $e_i=(x_i,x_{i+1})$ for $i=1,\cdots, n-1$, and $e_n=(x_n,\infty)$.
By picking a point $y_i\in e_i$ for $i=0,\cdots,n$, we obtain a full subcategory $Z$ of $\Ent(\R,\Sstrat)$ which has $2n+1$ objects and a single morphism from each $y_i$ to each adjacent $x_{j}$.
That is, $\Ent(\R, \Sstrat)$ is equivalent to the category $Z$ displayed below where each arrow corresponds to the unique equivalence class of entrance paths starting at $y_i$ and ending at $x_{i+1}$ or $x_i$, depending on which of the latter exist.
\[
Z: \qquad y_0 \rightarrow x_1 \leftarrow y_1 \rightarrow \cdots \leftarrow y_{n-1} \rightarrow x_n \leftarrow y_n
\]
We refer to $Z$ as the \define{zigzag category} associated to $\Ent(\R,\Sstrat)$.
Since the inclusion functor $Z \to \Ent(\R,\Sstrat)$ is fully faithful and essentially surjective, we conclude that $Z$ is equivalent to $\Ent(\R,\Sstrat)$.
\end{ex}

\begin{ex}\label{ex:C}
Suppose the complex plane $X=\CC$ is stratified into two strata $S=\CC - \{0\}$ and $S'=\{0\}$.
Choose a point $x \in S$.
Then the entrance path category $\Ent(X, \Sstrat)$ is equivalent to the finite subcategory $\Ccat$ consisting of two objects
$x$ and $0$,
a morphism $\sigma:x\to x$ that freely generates $\pi_1(\CC-\{0\},x) = \Hom_\Ccat(x,x)\cong \ZZ$ as a group,
and a unique morphism $\alpha \in \Hom_\Ccat(x,0)$.
In particular, $\alpha \ast \sigma=\alpha$.
\end{ex}

We now recall the following classical technical lemma, which will be useful in proving Propositions \ref{prop:terminal-object} and \ref{prop:groupoid} and Theorem \ref{thm:classification}.

\begin{lem}{\cite[Theorem 15.4]{munkres1984elements}}
\label{lem:refinement}
Let $X$ be a triangulable compact space and $\Ocat$ an open covering.
Then there is a triangulation $\Tstrat$ of $X$ such that every simplex in $\Tstrat$ lies in an element of $\Ocat$.
\end{lem}

Example \ref{ex:C} is an instance of a more general result, which we now introduce.

\begin{prop}\label{prop:terminal-object}
Let $(X,\Sstrat)$ be a conically stratified space.
Let $U$ be a basic open associated to a stratum $S\in\Sstrat$.
Any point $y\in U \cap S$ is terminal in the category $\Ent(U,\Sstrat_U)$, i.e.~there is only one equivalence class of entrance paths from any point $x\in U$ to $y\in U\cap S$.
\end{prop}
\begin{proof}
Associated to the basic open $U$ is a filtration-preserving homeomorphism $f : \Rspace^d \times C(L)
\to (U, \Sstrat_U)$
where $L$ is a filtered space.
For convenience, we forget about the map $f$ and identify $U$ with $\Rspace^d \times C(L)$.
Define the deformation retraction $\pi : \Rspace^d \times C(L) \times [0, 1] \to \Rspace^d \times C(L)$ as follows.
For a cone point $(u,\star) \in \Rspace^d \times C(L)$, let $\pi(u, \star, t) := (u, \star)$
for all $t$.
Any other point looks like $(u, (p, r), t)$ where $p \in L$ and $r > 0$.
In this case, let
	\begin{equation*}
	\pi (u, (p, r), t) :=
		\begin{cases}
		\big( u, (p, (1-t)r) \big) & \text{for $0 \leq t < 1$} \\
		(u, \star) & \text{for $t = 1$.}
		\end{cases}
	\end{equation*}

Consider any entrance path
$\alpha : (\Ispace, \Tstrat) \to \Rspace^d \times C(L)$
such that $\alpha(0) = x$ and $\alpha(1) = y$.
There is a canonical entrance path $\gamma$ from $x$ to $y$: start at $x$, follow the
deformation lines of $\pi$ down to $\pi(x,1) \in \Rspace^d \times \{\star\}$, and then take the straight
line to $y \in \Rspace^d \times \{\star\}$.
In other words, define $\gamma : (\Ispace, \{0,\sfrac{1}{2}, 1\}) \to \Rspace^d \times C(L)$
as
	\begin{equation*}
	\gamma(t) := \begin{cases}
	\pi ( \alpha(0), 2t ) & \text{for $0 \leq t < \sfrac{1}{2}$} \\
	(2-2t) \pi ( \alpha(0),1) + (2t-1) y & \text{for $\sfrac{1}{2} \leq t \leq 1$}.
	\end{cases}
	\end{equation*}
To prove our claim, it is enough prove that $\alpha$ is equivalent to $\gamma$.

Suppose $\Tstrat = \{t_0= 0, t_1, \cdots, t_n=1\}$ is the stratification of the domain of $\alpha$.
The projection of $\alpha$ along $\pi(\bullet,1)$ to the base stratum is an entrance path
$\beta := \pi(\alpha, 1)$ lying entirely in $\Rspace^d \times \{\star\}$.
Let us use the same domain $(\Ispace, \Tstrat)$ for $\beta$ even though $\Tstrat$ may not be a minimal
stratification.
Consider the following diagram, where each arrow corresponds to an entrance path:
\begin{equation*}
\begin{tikzcd}
\alpha(t_0) \ar[r] \ar[d, "\delta_0"] \ar[rd, ] & \alpha(t_1) \ar[rd] \ar[r] \ar[d, "\delta_1"]
& \alpha(t_2) \ar[r] \ar[d, "\delta_2"] \ar[rd, ]
& \cdots \ar[r] \ar[rd, ]  & \alpha(t_{n-2}) \ar[r] \ar[rd, ] \ar[d, "\delta_{n-2}"]
& \alpha(t_{n-1}) \ar[r] \ar[rd] \ar[d, "\delta_{n-1}"]
& \alpha(t_n) \ar[d, "\delta_n = y"] \\
\beta(t_0) \ar[r] & \beta(t_1) \ar[r] & \beta(t_2) \ar[r] & \cdots \ar[r] & \beta(t_{n-2}) \ar[r] & \beta(t_{n-1})
\ar[r] & \beta(t_n) .
\end{tikzcd}
\end{equation*}
Each arrow $\alpha(t_i) \to \alpha(t_{i+1})$ is the restriction of $\alpha$ from
$t_i$ to $t_{i+1}$, possibly reparametrized to be a path whose domain is $[0,1]$.
It is easy to see that the composition of the top row is equivalent to $\alpha$.
Each arrow $\beta(t_i) \to \beta(t_{i+1})$ is similarly defined to be an entrance path gotten by restricting $\beta$ from
$t_i$ to $t_{i+1}$ and reparametrizing, if necessary.
The composition of the bottom row is equivalent to $\beta$.
The deformation lines of $\pi$ induce a homotopy $\delta : (\Ispace, \Tstrat) \times [0,1] \to \Rspace^d \times C(L)$ defined as $\delta(t,s) := \pi ( \alpha(t), s )$ which takes $\alpha$ to $\beta$.
If the point $\alpha(t)$ belongs to a stratum $S'$, then $\pi ( \alpha(t), s )$ remains in $S'$ for all
$s \in [0,1)$.
Let $\delta_i : ( \Ispace, \{0,1\}) \to \Rspace^d \times C(L)$ be the entrance path
$\delta_i(s) := \pi ( \alpha(t_i), s )$.
We now have $n-1$ squares and each square maps to $\Rspace^d \times C(L)$ in a stratum
preserving way via $\delta$.
Now divide each square into two triangles by introducing the diagonal entrance path.
Each triangle is now an elementary homotopy.
By a finite number of elementary homotopies, $\alpha \approx \beta \ast \delta_0$.
The entrance path $\beta \ast \delta_0$ is almost our canonical entrance path $\gamma$.
If we write $\gamma = \gamma_2 \ast \gamma_1$ where $\gamma_1$ is the first half of the path
and $\gamma_2$ the second half, then $\gamma_1 \approx \delta_0$.
Both $\gamma_2$ and $\beta$ are entrance paths lying entirely in $\Rspace^d \times \{ \star \}$
between the same pair of points.
By the following Proposition~\ref{prop:groupoid} or simply because
$\Rspace^d \times \{ \star \}$ is contractible,
$\beta \approx \gamma_2$.
Therefore $\alpha \approx \gamma$.
\end{proof}

\begin{prop}
\label{prop:groupoid}
Suppose $(X, \Sstrat)$ is a conically stratified space with a unique stratum $X$.
Then the category $\Ent(X, \Sstrat)$ is equivalent to the fundamental groupoid $\pi_1(X)$.
\end{prop}
\begin{proof}
Recall the fundamental groupoid has points as objects and homotopy classes of paths as morphisms.
Thus any path is an entrance path.
Moreover, any two equivalent entrance paths are homotopic by definition of an equivalence.
It remains to be seen that every homotopy $H : \Ispace^2 \to X$ that is constant on $H_s(0)$ and $H_s(1)$
is an equivalence of entrance paths.
A cover of $X$ by open balls (i.e. basic opens) lifts to an open cover $\Ocat$ of $\Ispace^2$.
By Lemma \ref{lem:refinement}, there is a triangulation of $\Ispace^2$ such that each triangle maps
to an open ball in the covering.
Thus $H_0$ and $H_1$ are equivalent entrance paths.
\end{proof}

\begin{defn}\label{defn:constr-entrance}
A \define{representation of the entrance path category} is a functor $\Ffunc : \Ent(X, \Sstrat) \to \Omega$.
When $\Ent(X,\Sstrat)$ is equivalent to the fundamental groupoid $\pi_1(X)$, we call such a representation a \define{local system}.
A \define{morphism} $\phi : \Ffunc \to \Gfunc$ between two representations
is a natural transformation of functors.
Let $\big [ \Ent(X, \Sstrat), \Omega \big ]$ be the category of representations of $\Ent(X,\Sstrat)$.
\end{defn}

A morphism $\phi : \Ffunc \to \Gfunc$ between two representations of the entrance path category restricts to a morphism of local systems
over each stratum.
This is because for each stratum $S \in \Sstrat$, $\Ent(X, \Sstrat)$ restricted
to $S$ is the fundamental groupoid of $S$.
Furthermore, for each pair of points $s_1, s_2 \in S$ and any homotopy class
of paths $\alpha : s_1 \to s_2$, the following diagram commutes:
	\begin{equation*}
	\xymatrix{
	\Ffunc(s_1) \ar[rr]^{\Ffunc(\alpha)}_\cong \ar[d]_{\phi(s_1)} && \Ffunc(s_2) \ar[d]^{\phi(s_2)} \\
	\Gfunc(s_1) \ar[rr]_{\Gfunc(\beta)}^\cong && \Gfunc(s_2).
	}
	\end{equation*}

\begin{ex}[Representations of a Stratified Real Line]
  Let $(\R,\Sstrat)$ be the real line stratified into two vertices $x_0=0$ and $x_1=1$ and three 1-strata $e_0=(-\infty,0)$, $e_1=(0,1)$ and $e_2=(1,\infty)$.
  Consider the following two representations.
  First we define $\Ffunc:\Ent(\R,\Sstrat)\to\Vect_{\Bbbk}$ as follows:
  \begin{itemize}
    \item On objects: For $x\in [0,1]$, we set $\Ffunc(x):=\Bbbk$. For $x \in \R - [0,1]$, we set $\Ffunc(x):=0$.
    \item On morphisms: For any $x,y \in [0,1]$ and any entrance path $\alpha: x \rightarrow y$, we put $\Ffunc(\alpha):=\id$. For any other entrance path $\alpha$, we put $\Ffunc(\alpha):=0$.
  \end{itemize}
  Our second representation $\Gfunc:\Ent(\R,\Sstrat)\to\Vect_{\Bbbk}$ is defined as follows:
  \begin{itemize}
    \item On objects: For $x\in [0,1)$, we set $\Gfunc(x):=\Bbbk$. For $x \in \R - [0,1)$, we set $\Gfunc(x):=0$.
    \item On morphisms: For any $x,y \in [0,1)$ and any entrance path $\alpha: x \rightarrow y$, we put $\Gfunc(\alpha):=\id$.
    For any other entrance path $\alpha$, we put $\Gfunc(\alpha):=0$.
  \end{itemize}
  Consider the morphism $\phi:\Ffunc \to \Gfunc$ that is defined as $\phi(x):=\id_{\Bbbk}$ on $x\in [0,1)$ and $\phi(x):=0$ otherwise.
  To check that this is a valid morphism, consider the following commutative diagram, which illustrates the natural transformation $\phi$ over the support of $\Ffunc$:

  \begin{equation*}
  \begin{tikzcd}
  \Bbbk = \Ffunc(0) \arrow[dd,"\id_{\Bbbk}=\phi(0)"'] & \arrow[l,"\id_{\Bbbk}"'] \arrow[dd,"\id_{\Bbbk}=\phi(x)"'] \arrow[r, bend left, "\id_{\Bbbk}"] \Ffunc(x) & \arrow[l,bend left,"\id_{\Bbbk}"] \Ffunc(y) \arrow[r,"\id_{\Bbbk}"] \arrow[dd,"\phi(y)=\id_{\Bbbk}"] & \Ffunc(1) =\Bbbk \arrow[dd,"\phi(1)=0"] \\
  & & & \\
  \Bbbk = \Gfunc(0) & \arrow[l,"\id_{\Bbbk}"] \Gfunc(x) \arrow[r,bend left,"\id_{\Bbbk}"] & \arrow[l,bend left,"\id_{\Bbbk}"] \Gfunc(y) \arrow[r,"0"'] & \Gfunc(1)=0.
  \end{tikzcd}
  \end{equation*}

  Note that we have chosen two points $0 < x < y < 1$ arbitrarily in order to emphasize that there is an uncountable collection of points, i.e.~objects in $\Ent(\R,\Sstrat)$.
  However, for any pair of objects $x$ and $y$, there is at most one morphism from $x$ to $y$.
  Although $\Ffunc$ and $\Gfunc$ may seem like complicated diagrams in $\Vect_{\Bbbk}$, using Example \ref{ex:R}, both $\Ffunc$ and $\Gfunc$ are equivalent to representations of the equivalent zigzag subcategory $Z$ of $\Ent(\R,\Sstrat)$ with 5 objects
  \[
    y_0 \rightarrow x_0\leftarrow y_1 \rightarrow x_1 \leftarrow y_2.
  \]
  For example, $y_0=-1/2$, $x_0=0$, $y_1=1/2$, $x_1=1$ and $y_2=3/2$ will work.

  It is also convenient to summarize $\Ffunc$ and $\Gfunc$ via their colimits. The morphism $\phi:\Ffunc \to \Gfunc$ defined above will also induce a canonical morphism in $\Vect_{\Bbbk}$ from $\colim \Ffunc$ to $\colim \Gfunc$.
  We review this construction briefly.
  First we consider a candidate co-cone to $\Ffunc$:
  \begin{equation*}
  \begin{tikzcd}
    & & \Bbbk & & \\
    & & & & \\
  \Bbbk = \Ffunc(0) \arrow[uurr,"\id_{\Bbbk}"] & \arrow[l,"\id_{\Bbbk}"] \arrow[uur,"\id_{\Bbbk}"] \arrow[rr, bend left, "\id_{\Bbbk}"'] \Ffunc(x) & & \arrow[ll,bend left,"\id_{\Bbbk}"'] \Ffunc(y) \arrow[r,"\id_{\Bbbk}"'] \arrow[uul,"\id_{\Bbbk}"'] & \Ffunc(1) =\Bbbk. \arrow[uull,"\id_{\Bbbk}"']
  \end{tikzcd}
  \end{equation*}
  We leave it to the reader to verify that $\Bbbk$ is, in fact, the object that stands at the end of the initial co-cone in the category of co-cones, which we denote by $\colim \Ffunc$.
  By contrast, the object $\colim \Gfunc$ that stands at the end of the colimit co-cone is the zero vector space:
  \begin{equation*}
  \begin{tikzcd}
    & & 0 & & \\
    & & & & \\
  \Bbbk = \Gfunc(0) \arrow[uurr,"0"] & \arrow[l,"\id_{\Bbbk}"] \arrow[uur,"0"] \arrow[rr, bend left, "\id_{\Bbbk}"'] \Gfunc(x) & & \arrow[ll,bend left,"\id_{\Bbbk}"'] \Gfunc(y) \arrow[r,"0"'] \arrow[uul,"0"'] & \Gfunc(1) =0. \arrow[uull,"0"']
  \end{tikzcd}
  \end{equation*}
  Note that if we consider any co-cone to $\Gfunc$ and pre-compose the legs in that co-cone with the arrows that define $\phi$, then a co-cone to $\Gfunc$ stands as a co-cone to $\Ffunc$.
  Since the colimit of $\Ffunc$ is by definition the initial object in the category of co-cones to $\Ffunc$, there is a unique morphism from $\colim \Ffunc$ to $\colim \Gfunc$ that commutes with all the arrows involved.
  Finally, we note that one could simplify the computation of the colimit using the equivalent representation of $Z$; see Exercise 3.1.xii in~\cite{riehl2017category}.
\end{ex}

%%%%%%%%%%%%%%%%%%%%%%%%%%%%%%%%%%%%%%%%%%%%%%%%%
%%%%%%%%%%%%%%%%%%%%%%%%%%%%%%%%%%%%%%%%%%%%%%%%%
\section{van Kampen}
\label{sec:van_kampen}
%%%%%%%%%%%%%%%%%%%%%%%%%%%%%%%%%%%%%%%%%%%%%%%%%
%%%%%%%%%%%%%%%%%%%%%%%%%%%%%%%%%%%%%%%%%%%%%%%%%
The van Kampen theorem is best thought of as a proof that certain topological constructions are actually cosheaves.
In this section, we prove the van Kampen theorem for the entrance path category.
This will be key in proving that constructible cosheaves and representations of the entrance path category are equivalent.
The fact that representations of the entrance path category, along with their colimits, satisfy the cosheaf axiom is detailed in Section \ref{subsec:cosheaf-of-reps}.

Let $\Cat$ be the category of small categories.
This is the category whose objects are small categories and whose morphisms are functors.
Recall that $\Cat$ has all small limits and colimits, i.e.~$\Cat$ is a bi-complete category~\cite[Cor.~4.5.16]{riehl2017category}.

\begin{defn}
Let $\Ent : \Strat \to \Cat$ be the functor that assigns to each stratified space $(X, \Sstrat)$
the entrance path category $\Ent(X, \Sstrat)$
and to each stratum-preserving map $f : (Y, \Tstrat) \to (X, \Sstrat)$ the induced functor $\Ent(f) : \Ent(Y, \Tstrat) \to \Ent(X, \Sstrat)$.
\end{defn}

Since every open set in a stratified space inherits the structure of a stratified space, we have associated to a conically stratified space $X$ a functor $\Open(X) \to \Strat$ that takes $U$ to $(U,\Sstrat_U)$.
By post-composing this functor with $\Ent:\Strat \to \Cat$, we can think of the entrance path construction on $X$ as a pre-cosheaf of categories.
The van Kampen theorem says that this construction is actually a cosheaf.
The following argument proves the cosheaf axiom for complete covers of $X$, but it can be repeated verbatim for any open subset $U$ and any complete cover of $U$.
This proves that the entrance path construction is actually a cosheaf.

\begin{thm}[van Kampen]
\label{thm:vanKampen}
For a conically stratified space $(X, \Sstrat)$ and any complete cover $\Ucat$
of $X$, the universal functor, i.e.~morphism in $\Cat$, $\colim  \Ent |_{\cU} \to \Ent(X, \Sstrat)$ is an isomorphism.
\end{thm}

\begin{proof}
Let $\{ \Ffunc_i : \Ent(U_i, \Sstrat_i) \to \Ent(X, \Sstrat) \}$ be the set of functors
induced by the inclusion $(U_i, \Sstrat_i) \subseteq (X, \Sstrat)$, over all $U_i \in \Ucat$.
For each pair $U_i \subseteq U_j$, the following diagram commutes in $\Cat$:
	\begin{equation*}
	\xymatrix{
	\Ent(U_i, \Sstrat_i) \ar[rd]_{\Ffunc_i}
	\ar[rr]^{\Ent \big( (U_i , \Sstrat_i) \subseteq (U_j, \Sstrat_j) \big)} &&
	\Ent(U_j, \Sstrat_j) \ar[ld]^{\Ffunc_j} \\
	& \Ent(X, \Sstrat). &
	}
	\end{equation*}
The set of functors $\{ \Ffunc_i \}$ is a co-cone from the diagram
$\Ent |_{\Ucat}$ in $\Cat$.
We aim to show that this co-cone is initial and hence satisfies the universal property of the colimit.
To this end, let
$\{ \Gfunc_i : \Ent(U_i, \Sstrat_i) \to \Ccat  \}$
be any second co-cone from $\Ent |_{\Ucat}$.
We need to show there is a unique functor $\Qfunc : \Ent(X, \Sstrat) \to \Ccat$ such that for each $U_i \in \Ucat$, the following diagram commutes:
	\begin{equation*}
	\xymatrix{
	& \Ent(U_i , \Sstrat_i ) \ar[dl]_{\Ffunc_i} \ar[dr]^{\Gfunc_i} & \\
	\Ent(X, \Sstrat) \ar[rr]_{\Qfunc} & & \Ccat.
	}
	\end{equation*}

If such a functor $\Qfunc$ exists, then $\Qfunc(x)$ must be $\Gfunc_i(x)$
for every $U_i \in \Ucat$ containing $x$.
If $x \in U_i \cap U_j$, then there is
a $U_k \in \Ucat$ such that $x \in U_k \subseteq U_i \cap U_j$.
We have $\Gfunc_i(x) = \Gfunc_k(x) = \Gfunc_j(x)$.
For any $x \in X$, let $\Qfunc(x) := \Gfunc_i(x)$ where $U_i$
is any open set in $\Ucat$ containing $x$.

We now define $\Qfunc$ on the morphisms of $\Ent(X, \Sstrat)$.
Suppose an entrance path $\alpha$ lies entirely in an element $U_i \in \Ucat$.
Then let $\Qfunc(\alpha) := \Gfunc_i(\alpha)$.
Now consider an arbitrary entrance path $\alpha$.
By compactness of $\Ispace$, $\alpha$ can be written as a
finite composition $\alpha = \alpha_n \ast \cdots \ast \alpha_1$,
where each $\alpha_i$ lies in some element of $\Ucat$.
Let $\Qfunc(\alpha) := \Qfunc(\alpha_n) \circ \cdots \circ \Qfunc(\alpha_1)$.
This assignment is independent of the refinement of $\alpha$ as any two refinements
have a common refinement.

We now check that $\Qfunc$ sends an equivalence class of entrance paths
to the same morphism.
It suffices to check that for every elementary homotopy
$\alpha \Rightarrow \beta$, we have $\Qfunc(\alpha) = \Qfunc(\beta)$; see Definition~\ref{defn:elementary_homotopy} for a reminder.
Let $x = \alpha(0) = \beta(0)$ and $y = \alpha(1) = \beta(1)$.
Recall the map $h : \Delta \to (X, \Sstrat)$ in the definition of an elementary
homotopy.
One of the requirements on $h$ is that the point $y = h(v_2)$ has a basic open
neighborhood $U$ that contains the image of $h$.
Suppose $U \subseteq U_i$ for some $U_i \in \Ucat$.
Then $\alpha = \beta$ in $\Ent(U_i, \Sstrat_i)$ implying
$\Qfunc(\alpha) = \Qfunc(\beta)$.
Suppose there is no $U_i \notin \Ucat$ such that $U \subseteq U_i$.
Recall every basic open is locally cone-like and can be shrunk to fit inside any cover element
$U_i \in \Ucat$.
This means that there is a complete covering by basic opens $\mathcal{B}$ of $X$ such that for all
$B \in \mathcal{B}$ there exists a $U_i \in \Ucat$ such that $B \subseteq U_i$.
By Lemma \ref{lem:refinement}, we may subdivide $\Delta$ so that every triangle
is an elementary homotopy.
We may now write $\alpha \Rightarrow \beta$ as a finite sequence
of elementary relations.
We have $\Ffunc(\alpha) = \Ffunc(\beta)$.
\end{proof}

\subsection{Cosheaves of Representations}
\label{subsec:cosheaf-of-reps}

In this section we show that representations of the entrance path category, as well as their colimits, form cosheaves in a natural way.
First we introduce the category of representations of a stratified space $(X,\Sstrat)$.

\begin{defn}
Associated to each stratified space $(X, \Sstrat)$ is a category of representations valued in $\Omega$, written $\calR_{\Omega}(X,\Sstrat)$.
The objects and morphisms are as follows:
\begin{itemize}
	\item An object in $\calR_{\Omega}(X,\Sstrat)$ consists of a pair $(U,\Ffunc)$ where $U$ is an open set in $X$ and $\Ffunc : \Ent(U, \Sstrat_U) \to \Omega$ is a representation of the restriction of the entrance path category of $X$ to $U$.
	\item A morphism in $\calR_{\Omega}(X,\Sstrat)$ from $(U,\Ffunc)$ to $(V,\Gfunc)$ consists of the inclusion relation $U\subseteq V$ as well as a natural transformation $\Ffunc \Rightarrow \Gfunc|_{\Ent(U,\Sstrat_U)}$.
\end{itemize}
We write $\calR$ when $\Omega$ and $(X,\Sstrat)$ are clear from context.
\end{defn}

In the following definition and many of the upcoming proofs, we will adopt the convention that when referring to a representation of an entrance path category $\Gfunc$, the notation $\Gfunc|_U$ will stand as shorthand for $\Gfunc$ restricted to the subcategory $\Ent(U,\Sstrat_U)$, i.e.~$\Gfunc|_{\Ent(U,\Sstrat_U)}$.

\begin{defn}\label{defn:consistent-reps}
We say a functor $\Qfunc: \Open(X) \to \calR$ is a \define{consistent assignment of representations} if whenever $U\subseteq V$, the induced morphism $\Qfunc(U \subseteq V)$ from
$\Qfunc(U) = (U, \Ffunc)$ to $\Qfunc(V) = (V, \Gfunc)$ is the identity natural transformation
$\id: \Ffunc \Rightarrow \Gfunc |_U$ in the second slot.
\end{defn}

\begin{lem}\label{lem:cosheaf-of-reps}
If $\Qfunc: \Open(X) \to \calR$ is a consistent assignment of representations, then $\Qfunc$ is a cosheaf valued in $\calR$.
\end{lem}
\begin{proof}
First we note that $\calR$ has small colimits because it has coequalizers and coproducts;
the existence of coequalizers follows because pointwise the coequalizer of two natural transformations is a coequalizer in $\Omega$, which is a bi-complete category.

We will check the cosheaf axiom for just one open set, namely $X$,
because we can always repeat the argument below for an arbitrary open set.
Let $\cU$ be a complete cover of $X$.
First, note that any co-cone to $\Qfunc|_{\cU}$ defines, by virtue of Theorem~\ref{thm:vanKampen},
a representation of $\Ent(X,\Sstrat)$ in $\Omega$.
Given any co-cone $\Gfunc: \Ent(X,\Sstrat) \to \Omega$ to $\Qfunc|_{\cU}$, we must prove there exists
a unique natural transformation of functors $\eta: \Qfunc(X) \Rightarrow \Gfunc$.
Note that saying $\Gfunc$ is a co-cone to $\Qfunc|_{\cU}$ means that for each cover element $U\in \cU$ we have a natural transformation $\eta_U: \Qfunc(U) \Rightarrow \Gfunc|_U$.
We define a natural transformation $\eta:\Qfunc(X) \to \Gfunc$ pointwise by setting $\eta(x)=\eta_U(x)$ for any $U$ containing $x$.
Notice that if $x\in V\subseteq U$, then the hypothesis that $\Qfunc(U)|_V=\Qfunc(V)$ implies that $\eta_V(x)=\eta_U(x)$, so this definition is independent of the cover element chosen.

To prove that $\eta$ is actually a natural transformation, we have to prove that the following
diagram commutes for all equivalence classes of entrance paths $\alpha$ in $\Ent(X, \Sstrat)$:
	\begin{equation*}
	\begin{tikzcd}
	\Qfunc(X)(\alpha(0)) \ar[rr, "\eta(\alpha(0))"] \ar[d, "\Qfunc(X)(\alpha)"] && \Gfunc(\alpha(0))
	\ar[d, "\Gfunc(\alpha)"] \\
	\Qfunc(X)(\alpha(1)) \ar[rr, "\eta(\alpha(1))"] && \Gfunc(\alpha(1)).
	\end{tikzcd}
	\end{equation*}
This follows by an argument analogous to one used in the proof of Theorem~\ref{thm:vanKampen}.
We replace an entrance path $\alpha:x \to y$ with a composition of entrance paths $\alpha_n\ast \cdots \ast \alpha_1$, where each $\alpha_i$ is contained in some element $U_i$ of the cover $\cU$.
Since each $\eta_{U_i}$ is a natural transformation, each of the subsquares in the following diagram commutes.
\begin{equation*}
\begin{tikzcd}
\Qfunc(X)(x) \arrow[r, "\Qfunc(\alpha_1)"] \arrow[d,"\eta_{U_1}(x)"'] & \Qfunc(X)(x_1) \arrow[r] \arrow[d, "\eta_{U_1}(x_1)"'] & \cdots \arrow[r] & \Qfunc(X)(x_{n-1})  \arrow[r, "\Qfunc(\alpha_n)"] \arrow[d, "\eta_{U_n}(x_{n-1})"] & \Qfunc(X)(y) \arrow[d, "\eta_{U_n}(y)"] \\
\Gfunc(x) \arrow[r, "\Gfunc(\alpha_1)"] & \Gfunc(x_1) \arrow[r] & \cdots \arrow[r] & \Gfunc(x_{n-1}) \arrow[r, "\Gfunc(\alpha_n)"]
& \Gfunc(y)
\end{tikzcd}
\end{equation*}
This proves that $\eta:\Qfunc(X) \Rightarrow \Gfunc$ is a natural transformation.
Uniqueness follows from the fact that the natural transformation $\eta$ was defined using the natural transformations used in the co-cone from $\Qfunc|_{\cU}$ to $\Gfunc$.
Any other natural transformation must restrict to this same collection of natural transformations, which implies uniqueness.
This completes the proof.
\end{proof}

We now show that if we have a consistent assignment of representations, then replacing each representation with its colimit defines a cosheaf valued in $\Omega$.
To do this we introduce some more notation and a preparatory lemma.

We can define a functor $\Lfunc : \calR \to \Omega$ as follows:
To each representation $\Ffunc : \Ent(U, \Sstrat_U) \to \Omega$, let $\Lfunc (U, \Ffunc) := \colim \Ffunc$.
To define $\Lfunc$ on morphisms we use functoriality of colimits, which we review briefly in this special case.
Recall that each morphism $(U, \Ffunc) \to (V, \Gfunc)$ in $\calR$ specifies a natural transformation $\eta: \Ffunc \Rightarrow \Gfunc|_U$.
The colimit provides a natural transformation $\Delta_{\Gfunc}$ from $\Gfunc$ to the constant diagram with value $\colim \Gfunc$.
We can restrict the domain of definition of the natural transformation $\Delta_{\Gfunc}$ to define a natural transformation from $\Gfunc|_U$ to $\colim \Gfunc$.
We can then pre-compose this restricted natural transformation with $\eta$ to make $\colim \Gfunc$ a co-cone to $\Ffunc$.
The universal property of $\colim \Ffunc$ then provides a unique morphism to $\colim \Gfunc$ restricted along $\eta$.
We define $\Lfunc (U, \Ffunc) \to \Lfunc(V, \Gfunc)$ to be this unique morphism.

Now consider a functor $\Rfunc : \Omega \to \calR$ in the opposite direction.
For each object $a$ in $\Omega$, let $\Rfunc(a)$ be the $a$-constant representation
$\Afunc: \Ent(X, \Sstrat) \to \Omega$.
That is, $\Afunc(x) = a$ for all points $x$ and $\Afunc(\alpha) = \id$ for all entrance paths $\alpha$.
For every morphism $a \to b$ in $\Omega$, let $\Rfunc(a \to b)$
be the induced natural transformation $\Afunc \Rightarrow \Bfunc$.
We now have the following diagram of categories and functors:
	\begin{equation*}
	\begin{tikzcd}
	\Omega  \arrow[rr, bend left, "\Rfunc"] && \calR \arrow[ll, "\Lfunc", bend left]
	\end{tikzcd}
	\end{equation*}

\begin{lem}\label{lem:left-adjoint}
The functor $\Lfunc$ is left adjoint to $\Rfunc$.
\end{lem}
\begin{proof}
Recall that for any category $\Ccat$, $1_{\Ccat}$ is the identity functor on $\Ccat$.
We construct the co-unit $\epsilon : \Lfunc \circ \Rfunc \Rightarrow 1_\Omega$ of the adjunction as follows.
For an object $a$ in $\Omega$, recall that $\Lfunc \circ \Rfunc (a)$ is the colimit of the $a$-constant
representation $\Afunc : \Ent(X, \Sstrat) \to \Omega$,
which is canonically isomorphic to $a$.
Let $\epsilon(a)$ be the universal morphism from the colimit to $a$.
Note that whenever we have a morphism $a \to b$ in $\Omega$, the following square commutes:
\begin{equation*}
\begin{tikzcd}
\Lfunc \circ \Rfunc (a) \arrow[r] \arrow[d, "\epsilon(a)"'] & \Lfunc \circ \Rfunc (b) \arrow[d, "\epsilon(b)"] \\
a \arrow[r] & b.
\end{tikzcd}
\end{equation*}

Now we turn to the construction of the unit $\eta : 1_\calR \Rightarrow \Rfunc \circ \Lfunc$.
Recall an object of $\calR$ is a pair $(U, \Ffunc)$ where $U \subseteq X$ is an open set
and $\Ffunc : \Ent(U, \Sstrat_U) \to \Omega$ is a functor.
The composition $\Rfunc \circ \Lfunc (U, \Ffunc)$ is the $b$-constant functor
$\Bfunc : \Ent(U, \Sstrat_U) \to \Omega$ where $b$ is the colimit of $\Ffunc$.
Define $\eta(U, \Ffunc)$ to be the natural transformation from $\Ffunc$ to the constant diagram $\Bfunc$ given by the definition of the colimit.
One can easily check the commutativity of the following diagrams:
	\begin{equation*}
	\begin{tikzcd}
	\Lfunc \arrow[r, "\Lfunc \eta", Rightarrow] \arrow[rd, "\id_{\Lfunc}"', Rightarrow]
	& \Lfunc \Rfunc \Lfunc \arrow[d, "\epsilon \Lfunc", Rightarrow]
	&&& \Rfunc \arrow[r, "\eta \Rfunc", Rightarrow] \arrow[rd, "\id_{\Rfunc}"', Rightarrow]&
	\Rfunc \Lfunc \Rfunc \arrow[d, "\Rfunc \epsilon", Rightarrow] \\
	& \Lfunc &&& & \Rfunc
	\end{tikzcd}
	\end{equation*}
\end{proof}

\begin{lem}
\label{lem:adjoint}
Let $\Qfunc : \Open(X) \to \calR$ be a consistent assignment of representations, as defined in Definition~\ref{defn:consistent-reps}.
Let $\Lfunc\circ \Qfunc:\Open(X) \to \Omega$ be the functor that assigns to each open $U$ the colimit of $\Qfunc(U):\Ent(U,\Sstrat_U) \to \Omega$.
The pre-cosheaf $\Lfunc\circ\Qfunc$ is a cosheaf.
\end{lem}
\begin{proof}
Suppose $\cU$ is a complete cover of an open set $U$.
We need to show that the universal morphism
\[
\colim \Lfunc\circ \Qfunc|_{\cU} \to \Lfunc\circ\Qfunc(U)
\]
is an isomorphism.
By Lemma~\ref{lem:left-adjoint}, we know that $\Lfunc$ is a left adjoint and hence commutes with colimits so the map
\[
\colim \Lfunc\circ \Qfunc|_{\cU} \to \Lfunc(\colim \Qfunc|_{\cU})
\]
is an isomorphism.
Moreover, by Lemma~\ref{lem:cosheaf-of-reps} we know that $\colim \Qfunc|_{\cU} \to \Qfunc(U)$
is an isomorphism as well.
Post-composing this map with $\Lfunc$ and using the fact that functors preserve isomorphisms, we have that the desired composition
\[
	\colim \Lfunc\circ \Qfunc|_{\cU} \to \Lfunc(\colim \Qfunc|_{\cU}) \to \Lfunc(\Qfunc(U)) = \Lfunc\circ\Qfunc(U)
\]
is an isomorphism as well.
\end{proof}

%%%%%%%%%%%%%%%%%%%%%%%%%%%%%%%%
\section{Equivalence}
\label{sec:equivalence}
%%%%%%%%%%%%%%%%%%%%%%%%%%%%%%%%
In this section we prove the main equivalence of this paper, which was originally observed by Robert MacPherson.
We recall the basic ingredients of this equivalence.
Given a conically stratified space $(X, \Sstrat)$
we have the notion of an $\Sstrat$-constructible cosheaf given in Definition~\ref{defn:constr-cosheaf} and a representation of the entrance path category given in Definition~\ref{defn:constr-entrance}.
These are the objects of the categories $\Csh_{\Omega}(X,\Sstrat)$ and $\big [ \Ent(X, \Sstrat), \Omega \big ]$.
The main theorem of this paper is that these two categories are equivalent.

\begin{thm}[Classification]
\label{thm:classification}
$\Csh_\Omega(X, \Sstrat)$ is equivalent to $\big [ \Ent(X, \Sstrat), \Omega \big ]$.
\end{thm}

\begin{proof}
Recall that the category $\Csh_\Omega (X, \Sstrat)$ is \emph{equivalent}
\footnote{Some authors say two categories are equivalent if there is a single functor between them that is fully faithful and essentially surjective. The definition of equivalence we use follows the one given in Riehl's \emph{Category Theory in Context}, see~\cite[Def.~1.5.4]{riehl2017category}. This notion of equivalence resembles the notion of an {\em adjoint equivalence}, but it does not require the extra triangle identities that an adjoint equivalence needs to satisfy. However, this is fine because~\cite[Prop.~4.4.5]{riehl2017category} proves that every equivalence in our sense can be promoted to an adjoint equivalence.} to
the category $\big [ \Ent(X, \Sstrat), \Omega \big ]$ if there are functors
$$\Iffunc : \Csh_\Omega(X, \Sstrat) \to \big [ \Ent(X, \Sstrat), \Omega \big ]$$
$$\Jffunc : \big [ \Ent(X, \Sstrat), \Omega \big ] \to \Csh_\Omega(X, \Sstrat)$$
and natural isomorphisms
\begin{align*}
\epsilon:\Jffunc\circ\Iffunc \Rightarrow \id_{\Csh_\Omega(X, \Sstrat)} &&
\eta: \id_{[ \Ent(X, \Sstrat), \Omega]} \Rightarrow \Iffunc\circ\Jffunc.
\end{align*}
The natural transformation $\epsilon$ is called the co-unit and
the natural transformation $\eta$ is the unit.
We now construct this equivalence.

\paragraph{Construction of $\Iffunc$.}
We start with $\Iffunc$.
Fix an $\Sstrat$-constructible cosheaf $\cosheaf{F}$.
Recall that $\cosheaf{F}(x)$ is the costalk of $\cosheaf{F}$ at the point $x \in X$.
On objects $x \in \Ent(X,\Sstrat)$, let
$$\Iffunc(\cosheaf{F})(x) := \cosheaf{F}(x).$$
To describe how $\Iffunc(\Ffunc)$ acts on an entrance path $\alpha: x \to y$, we first describe how it acts on an entrance path $\alpha: (\Ispace,\Tstrat) \to (X,\Sstrat)$ that intersects only two strata in the following way:
\[
	\exists \, S_0, S_1\in \Sstrat \qquad \text{s.t.} \qquad \alpha([0,1))\subseteq S_0 \qquad \text{and} \qquad \alpha(1)=y\in S_1.
\]
We further require that there is a basic open $U$ whose cone point is $\alpha(1)=y$ that contains $\alpha(0)=x$. Note that we allow the possibility that $S_0=S_1$.
% Let $\alpha : x \to y$ be an entrance path with the property that there is a basic open neighborhood
% $U \subseteq X$ of $y$ containing $x$.
In either case, the diagram of costalk maps
\[
	\begin{tikzcd}
	\cosheaf{F}(x) \arrow[r] & \cosheaf{F}(U) & \cosheaf{F}(y) \arrow[l,"\cong"]
	\end{tikzcd}
\]
allows us to define $\Iffunc(\cosheaf{F})(\alpha)$ as $\Ffunc^{-1}(y\in U)\circ \Ffunc(x\in U)$.
The morphism $\Iffunc(\cosheaf{F})(\alpha)$ is independent of the choice of $U$.
This is because for any other basic open neighborhood $V$ of $y$ containing $x$
where $U \subseteq V$, the following diagram commutes:
\[
	\begin{tikzcd}
	& \cosheaf{F}(V) & \\
	\cosheaf{F}(x) \arrow[ru] \arrow[rd] &  & \cosheaf{F}(y) \arrow[lu,"\cong"'] \arrow[ld,"\cong"] \\
	& \cosheaf{F}(U). \arrow[uu, "\Ffunc(U \subseteq V)", "\cong"'] &
	\end{tikzcd}
\]
Finally, we note that an arbitrary entrance path is equivalent to a finite composition of entrance paths with the above properties.
The morphism for this arbitrary entrance path will be a zig-zag of costalk maps, where left-pointing arrows are isomorphisms.

To show that this definition is invariant under our equivalence relation it suffices to check invariance under an elementary homotopy.
This will also show that when the composition of entrance paths is related to another entrance path by an elementary homotopy, then the above construction is well-defined.
Consider an elementary homotopy $\alpha \Rightarrow \beta$
and recall the associated map $h : \Delta \to (X, \Sstrat)$ from
Definition~\ref{defn:elementary_homotopy}.
Note that there is a basic open
$U_2 \in \Basic(X, \Sstrat)$ with cone point $z:=h(v_2)$ that contains the entire image of $h$.
Additionally, we can guarantee the existence of a basic open $U_1\subseteq U_2$ with cone point $y:=h(v_1)$.
Finally, let $x=h(v_0)$.
In this setting, $\beta: x \to z$ is an entrance path with the property mentioned above: it only intersects two strata and has a basic open around its endpoint.
Following the construction given above, we define $\Iffunc(\cosheaf{F})(\beta)$ as the following composition:
\begin{equation*}
	\xymatrix{
	\cosheaf{F}(\beta(0)) \ar[rr]^-{\cosheaf{F}(x \in U_2)}
	&& \cosheaf{F}(U_2) \ar[rrr]^-{\cosheaf{F}^{-1}(z \in U_2)}_-\cong &&&  \cosheaf{F}(\beta(1)).
	}
\end{equation*}
On the other hand, $\alpha$ can be viewed as the composition of two entrance paths $\alpha_1:x \to y$ and $\alpha_2: y \to z$ with the above properties.
Our construction above tells us to define $\Iffunc(\cosheaf{F})(\alpha):=\Iffunc(\cosheaf{F})(\alpha_2)\circ \Iffunc(\cosheaf{F})(\alpha_1)$, which is the following composition of morphisms in $\Omega$:
\[
	\xymatrix{
	\cosheaf{F}(\alpha(0)) \ar[rr]^-{\cosheaf{F}(x \in U_1)} && \cosheaf{F}(U_1) \ar[rr]^-{\cosheaf{F}^{-1}(y \in U_1)} && \cosheaf{F}(\alpha(1/2)) \ar[rr]^-{\cosheaf{F}(y \in U_2)} && \cosheaf{F}(U_2) \ar[rr]^-{\cosheaf{F}^{-1}(z \in U_2)} && \cosheaf{F}(\alpha(1)).
	}
\]
However since the co-stalk $\cosheaf{F}(\alpha(1/2))$ defines a universal cone to $\cosheaf{F}$ restricted to the collection of open sets containing $\alpha(1/2)=y$, the middle two arrows cancel to yield
\[
	\xymatrix{
	\cosheaf{F}(\alpha(0)) \ar[rr]^-{\cosheaf{F}(x \in U_1)} && \cosheaf{F}(U_1) \ar[rr]^-{\cosheaf{F}(U_1 \subseteq U_2)} && \cosheaf{F}(U_2) \ar[rr]^-{\cosheaf{F}^{-1}(z \in U_2)} && \cosheaf{F}(\alpha(1)).
	}
\]
This composition further reduces to
\[
	\xymatrix{
	\cosheaf{F}(\alpha(0)) \ar[rr]^-{\cosheaf{F}(x \in U_2)} && \cosheaf{F}(U_2) \ar[rr]^-{\cosheaf{F}^{-1}(z \in U_2)} && \cosheaf{F}(\alpha(1)),
	}
\]
which is exactly the definition of $\Iffunc(\cosheaf{F})(\beta)$.
Thus we have just shown that
$\Iffunc (\cosheaf{F})(\alpha) = \Iffunc (\cosheaf{F})(\beta)$
whenever $\alpha \Rightarrow \beta$.
Repeating this argument finitely many times shows that for any two entrance paths $\alpha$ and $\beta$ such that $\alpha \approx \beta$, we have $\Iffunc (\cosheaf{F})(\alpha) = \Iffunc ( \cosheaf{F})(\beta)$.

A cosheaf morphism $\phi : \cosheaf{F} \to \cosheaf{G}$ induces
a morphism of representations
$\Iffunc(\phi) : \Iffunc(\cosheaf{F}) \to \Iffunc (\cosheaf{G})$
by virtue of the fact that a map of cosheaves is a natural transformation indexed by open sets of $X$, which in turn induces maps on the level of costalks.
To check that that the morphism $\Iffunc(\phi)$ is a natural transformation of representations, we observe that the morphism associated to an entrance path is determined by a cover of the entrance path by basic opens, where $\phi$ already defines a natural transformation.
Because $\cosheaf{F}$ and $\cosheaf{G}$ are functors on $\Open(X)$, we can always refine two choices of covers of an entrance path to a common third one to guarantee that this map of representations is well-defined.

\paragraph{Construction of $\Jfunc$.}
Now we construct $\Jfunc$.
Fix a representation $\Ffunc : \Ent(X, \Sstrat) \to \Omega$.
Following the notation in Section~\ref{subsec:cosheaf-of-reps}, we can associate to $\Ffunc$ a consistent assignment of representations $\Qfunc:\Open(X) \to \calR$ that assigns to each open set $U$, the object $(U,\Ffunc|_{\Ent(U, \Sstrat_U)})$.
Recall that $\Lfunc:\calR \to \Omega$ assigns to each representation its corresponding colimit.
Define $\Jffunc(\Ffunc)=\Lfunc\circ \Qfunc$.
This means that to each open $U \subseteq X$
\[
\Jffunc(\Ffunc)(U) := \colim \Ffunc |_{\Ent(U, \Sstrat_U)}
\]
and to each pair of open sets $U \subseteq V$,
$$\Jffunc(\Ffunc)(U \subseteq V) : \Jffunc(\Ffunc)(U) \to \Jffunc(\Ffunc)(V)$$
is the universal morphism between the two colimits.
We already proved in Lemma~\ref{lem:adjoint} that $\Jffunc(\Ffunc)=\Lfunc \circ \Qfunc$ satisfies the cosheaf axiom.
To check constructibility, let $U \subseteq V$ be a pair in $\Basic(X, \Sstrat)$ associated to some common stratum $S \in \Sstrat$.
Choose a point $x \in U \cap V \cap S$.
By Proposition~\ref{prop:terminal-object}, $x$ is terminal in both $\Ent(U, \Sstrat_U)$ and $\Ent(V, \Sstrat_V)$.
Since the colimit of a diagram with a terminal object is just the value given to the terminal object, we get the following isomorphisms along with a morphism induced by inclusion:
	\begin{equation*}
	\begin{tikzcd}
	\Ffunc(x) \cong \colim \Ffunc|_{\Ent(U,\Sstrat_{U})} \ar[rr, "\cong"] &&
	\colim \Ffunc|_{\Ent(V,\Sstrat_{V})} \cong \Ffunc(x).
	\end{tikzcd}
	\end{equation*}

A morphism of representations $\phi: \Ffunc \to \Gfunc$ of $\Ent(X,\Sstrat)$ is a natural transformation of functors, which is defined pointwise for $x\in X$.
Consequently, any morphism $\phi$ restricts to a morphism $\phi|_U:\Ffunc|_{\Ent(U,\Sstrat_U)} \to \Gfunc|_{\Ent(U,\Sstrat_U)}$ of representations of $\Ent(U,\Sstrat_U)$, which in turn specifies a morphism between the associated colimits, i.e. $\Jffunc(\Ffunc)(U) \to \Jffunc(\Gfunc)(U)$.
Finally, it is easy to see that for a pair of open sets $U\subseteq V$ the restriction of $\phi|_V$ to $\Ent(U,\Sstrat_U)$ agrees with $\phi|_U$, so we have the following commutative square
\begin{equation*}
	\xymatrix{
	\Jffunc(\cosheaf{F})(V) \ar[rr]
	&& \Jffunc(\cosheaf{G})(V) \\
	\Jffunc(\cosheaf{F})(U) \ar[rr] \ar[u] && \Jffunc(\cosheaf{G})(U) \ar[u]
	}
\end{equation*}
thereby implying that a morphism of representations is sent to a morphism of cosheaves.

\paragraph{Construction of $\epsilon$.}
We now construct the co-unit $\epsilon:\Jffunc\circ\Iffunc \Rightarrow \id_{\Csh_\Omega(X, \Sstrat)}$.
Fix a constructible cosheaf $\cosheaf{F}$.
Recall that the associated representation $\Iffunc(\cosheaf{F})$ of the entrance path category is turned back into a cosheaf by taking the colimit of the restriction of this representation to each open set.
That is to say
\[
	\Jffunc\circ\Iffunc(\cosheaf{F})(U):=\colim \Iffunc(\cosheaf{F})|_{\Ent(U, \Sstrat_U)}.
\]
To specify the co-unit $\epsilon(\cosheaf{F})(U)$, we note that for each $x\in U$, $\Iffunc(\cosheaf{F})(x)$ is defined as the co-stalk of $\cosheaf{F}$ at $x$, which provides a morphism $\Iffunc(\cosheaf{F})(x) \to \cosheaf{F}(U)$.
We claim that these morphisms collate to form a co-cone $\chi_U$ to the restricted representation $\Iffunc(\cosheaf{F})|_{\Ent(U, \Sstrat_U)}$, which we write as $\Iffunc(\cosheaf{F})|_{U}$ for short.
The fact that the costalk maps form a co-cone $\chi_U$ follows from the observation that the morphism $\Iffunc(\cosheaf{F})(\alpha)$ associated to any entrance path $\alpha$ in $\Ent(U, \Sstrat_U)$ is defined using a zig-zag of costalk maps involving basic opens, each of which can be chosen to fit inside of $U$.
Since $\cosheaf{F}(U)$ forms a co-cone to $\Iffunc(\cosheaf{F})|_{U}$, we get the unique morphism from the colimit
\[
\Jffunc\circ\Iffunc(\cosheaf{F})(U):=\colim \Iffunc(\cosheaf{F})|_{\Ent(U, \Sstrat_U)} \to \cosheaf{F}(U).
\]
This unique morphism defines our co-unit $\epsilon(\cosheaf{F})(U)$ for each open set $U$.
In order to show that this assignment of morphisms to each open set collectively defines a morphism of cosheaves
$\epsilon(\cosheaf{F}):\Iffunc\circ\Jffunc (\cosheaf{F}) \to \cosheaf{F}$
we must show that for any pair of open sets $U \subseteq V$, the following diagram commutes:
	\begin{equation*}
	\xymatrix{
	\Jffunc \circ \Iffunc (\cosheaf{F})(U) \ar[rr]^{\Jffunc \circ
	\Iffunc (\cosheaf{F})(U \subseteq V)} \ar[d]_{\epsilon(\cosheaf{F})(U)}
	&& \Jffunc \circ \Iffunc (\cosheaf{F})(V) \ar[d]^{\epsilon(\cosheaf{F})(V)} \\
	\cosheaf{F}(U) \ar[rr]^{\cosheaf{F}(U \subseteq V)} && \cosheaf{F}(V).
	}
	\end{equation*}
In order to prove that the above square commutes, we introduce some additional notation and describe the construction of $\epsilon$ in more detail.
The high-level structure of the proof is that since $\Jffunc \circ \Iffunc (\cosheaf{F})(U)$ is defined in terms of a colimit, we want to use uniqueness of the induced morphism to $\cosheaf{F}(V)$ to prove commutativity, but this requires showing that the two ways of going around the square can be used to define the same co-cone to $\Iffunc(\cosheaf{F})|_U$.

%requires some additional notation and a recollection of how, precisely, restrictions of representations define the induced morphisms between colimits. Details are provided below, but the reader may want to see the second paragraph above Lemma~\ref{lem:left-adjoint} for more context.
Recall that a co-cone is just a natural transformation of functors, where the the codomain of the transformation is a functor with constant value.
Let $\Delta_U:\Iffunc(\cosheaf{F})|_U \Rightarrow \Jffunc \circ \Iffunc (\cosheaf{F})(U):=\colim \Iffunc(\cosheaf{F})|_U$ denote the co-cone (natural transformation) from the representation $\Iffunc(\cosheaf{F})|_U$ to the constant representation with value the colimit of $\Iffunc(\cosheaf{F})|_U$.
Let $\chi_U:\Iffunc(\cosheaf{F})|_U \Rightarrow \cosheaf{F}(U)$ be the co-cone described above, which uses the co-stalk maps $\{\cosheaf{F}(x) \to \cosheaf{F}(U)\}_{x\in U}$.
Define $\Delta_V$ and $\chi_V$ similarly by replacing every occurrence of $U$ in the past two sentences with $V$.
By construction, we have that
\[
	\epsilon(\cosheaf{F})(U) \circ \Delta_U =\chi_U \qquad \text{and} \qquad \epsilon(\cosheaf{F})(V) \circ \Delta_V =\chi_V.
\]
Recall that the morphism $\Jffunc\circ\Iffunc(\cosheaf{F})(U\subseteq V)$ is defined by taking the co-cone $\Delta_V: \Iffunc(\cosheaf{F})|_V \Rightarrow \colim \Iffunc(\cosheaf{F})|_V$ and restricting the domain of definition to $\Ent(U,\Sstrat_U)$ to get a new co-cone $\Delta_V|_U : \Iffunc(\cosheaf{F})|_U \Rightarrow \colim \Iffunc(\cosheaf{F})|_V$.
The morphism $\Jffunc\circ\Iffunc(\cosheaf{F})(U\subseteq V)$ is the unique morphism from $\colim \Iffunc(\cosheaf{F})|_U$ to $\colim \Iffunc(\cosheaf{F})|_V$ satisfying
\[
	\left(\Jffunc\circ\Iffunc(\cosheaf{F})(U\subseteq V)\right) \circ \Delta_U = \Delta_V|_U.
\]
Finally, we note that post-composing the natural transformation $\chi_U$ with the morphism $\cosheaf{F}(U\subseteq V)$ can be identified with the restricted natural transformation $\chi_V|_U$, i.e.
$$\cosheaf{F}(U\subseteq V)\circ \chi_U = \chi_V|_U.$$
We now use the above identities to show that
\begin{eqnarray*}
	\epsilon(\cosheaf{F})(V)\circ \left(\Jffunc\circ\Iffunc(\cosheaf{F})(U\subseteq V)\right) \circ \Delta_U &=& \epsilon(\cosheaf{F})(V) \circ \Delta_V|_U \\
	& = & \chi_V|_U \\
	& = & \cosheaf{F}(U\subseteq V) \circ \chi_U \\
	& = & \cosheaf{F}(U\subseteq V) \circ \epsilon(\cosheaf{F})(U) \circ \Delta_U
\end{eqnarray*}
This almost shows what we want, except for the factor $\Delta_U$, but the identity above is crucial because we have two ways of writing the same co-cone from $\Iffunc(\cosheaf{F})|_U$ to $\cosheaf{F}(V)$.
By universal properties, we have a \emph{unique} morphism $\psi_{U,V}:\colim \Iffunc(\cosheaf{F})|_U =: \Jffunc \circ \Iffunc(\cosheaf{F})(U) \to \cosheaf{F}(V)$ satisfying
\[
	\cosheaf{F}(U\subseteq V) \circ \epsilon(\cosheaf{F})(U) \circ \Delta_U = \psi_{U,V} \circ \Delta_U = \epsilon(\cosheaf{F})(V)\circ \left(\Jffunc\circ\Iffunc(\cosheaf{F})(U\subseteq V)\right) \circ \Delta_U.
\]
Uniqueness of $\psi_{U,V}$ implies the final, desired equality:
$$\cosheaf{F}(U\subseteq V) \circ \epsilon(\cosheaf{F})(U)= \psi_{U,V} = \epsilon(\cosheaf{F})(V)\circ \left(\Jffunc\circ\Iffunc(\cosheaf{F})(U\subseteq V)\right).$$

All of the above arguments are natural, meaning that for every cosheaf morphism $\phi : \cosheaf{F} \to \cosheaf{G}$, the cosheaf morphisms given by $\epsilon$ make the following diagram commute
	\begin{equation*}
	\xymatrix{
	\Jffunc \circ \Iffunc(\cosheaf{F}) \ar[d]_{\epsilon(\cosheaf{F})}
	\ar[rr]^{\Jffunc \circ \Iffunc( \phi )}
	&& \Jffunc \circ \Iffunc (\cosheaf{G}) \ar[d]^{\epsilon(\cosheaf{G})} \\
	\cosheaf{F} \ar[rr]^{\phi} && \cosheaf{G}.
	}
	\end{equation*}
In other words $\epsilon:\Jffunc\circ\Iffunc \Rightarrow \id_{\Csh_\Omega(X, \Sstrat)}$ is an actual natural transformation between functors defined on the category of constructible cosheaves.

All that remains to be shown is that for each $\cosheaf{F}$, the co-unit $\epsilon(\cosheaf{F})$ is an isomorphism of cosheaves.
We first show that for each \emph{basic} open $U$, $\epsilon(\cosheaf{F})(U)$ is an isomorphism, we then invoke the cosheaf axiom to show this implies that $\epsilon(\cosheaf{F})(U)$ is an isomorphism on arbitrary open sets.
First, if $U$ is a basic open associated to a stratum $S$, then Proposition~\ref{prop:terminal-object} shows that any $x\in U\cap S$ is terminal in $\Ent(U,\Sstrat|_U)$.
It is an easy exercise to show that if an indexing category for a functor has a terminal object, then the colimit of that functor is just the value of the functor on the terminal object.
This implies that for any $x\in U\cap S$, the morphism $\Iffunc(\cosheaf{F})(x) \to \Jffunc \circ \Iffunc (\cosheaf{F})(U)$ is an isomorphism.
Constructibility of $\cosheaf{F}$ implies we have the following commutative triangle:
\[
	\xymatrix{\Jffunc \circ \Iffunc (\cosheaf{F})(U) \ar[rr]^{\epsilon(\cosheaf{F})(U)} && \cosheaf{F}(U) \\
	& \Iffunc(\cosheaf{F})(x) \ar[lu]^{\cong} \ar[ru]_{\cong} &}
\]
This implies the top horizontal arrow is an isomorphism as well and more generally that the restricted natural transformation
\[
\epsilon(\cosheaf{F})|_{\Basic(U,\Sstrat_U)} : \Jffunc\circ\Iffunc(\cosheaf{F})|_{\Basic(U,\Sstrat_U)} \Rightarrow \cosheaf{F}|_{\Basic(U,\Sstrat_U)}
\]
is a natural isomorphism.
Finally, for an arbitrary open set $U$, the collection of basic opens inside of $U$ forms a complete cover of $U$.
The cosheaf axiom (Definition~\ref{defn:cosheaf}) implies that the vertical arrows below are isomorphisms.
	\begin{equation*}
	\xymatrix{
	\Jffunc \circ \Iffunc (\cosheaf{F})(U) \ar[rr]^{\epsilon(\cosheaf{F})(U)}
	&& \cosheaf{F}(U) \\
	\colim \Jffunc\circ\Iffunc(\cosheaf{F})|_{\Basic(U,\Sstrat_U)} \ar[u]^{\cong} \ar[rr]_{\cong} && \colim \cosheaf{F}|_{\Basic(U,\Sstrat_U)} \ar[u]_{\cong}
	}
	\end{equation*}
Since the restriction of $\epsilon(F)$ to $\Basic(U,\Sstrat_U)$ defines a natural isomorphism, the induced map on colimits is also an isomorphism, which proves the bottom horizontal arrow is an isomorphism.
This proves that $\epsilon(\cosheaf{F})(U)$ is an isomorphism for an arbitrary open set $U$ and hence that $\epsilon(\cosheaf{F}): \Jffunc \circ \Iffunc (\cosheaf{F}) \to \cosheaf{F}$ is a natural isomorphism.

\paragraph{Construction of $\eta$.}
We now define the unit $\eta$ on points of $X$ and elementary homotopies.
Fix a functor $\Ffunc: \Ent(X, \Sstrat) \to \Omega$.
Recall that $\Jffunc(\Ffunc)$ is a constructible cosheaf.
$\Iffunc \circ \Jffunc (\Ffunc)(x)$ is then defined as the inverse limit of $\Jffunc(\Ffunc)$ over all open sets containing $x$.
Note that the collection of \emph{basic opens} containing $x$ forms an initial (also called final) subsequence, so restricting the limit to basic opens yields an isomorphic inverse limit.
Let $U$ be any basic open neighborhood of $x$.
As we saw in the construction of $\epsilon$ above,
Proposition~\ref{prop:terminal-object}
guarantees that the morphism
$z_U: \Ffunc(x) \to \colim \Ffunc|_{\Ent(U, \Sstrat_U)}=:\Jffunc(\Ffunc)(U)$
used in that universal co-cone is an isomorphism.
% $z_U : \Jffunc(\Ffunc)(U):=\colim \Ffunc|_{\Ent(U, \Sstrat_U)} \to \Ffunc(x)$ that is also an isomorphism.
This means that the collection of arrows $\{z_U : \Ffunc(x) \to \Jffunc(\Ffunc)(U)\}$ ranging over all basic opens $U$ containing $x$ defines a limit cone to the restriction of $\Jffunc(\Ffunc)$ to basic opens containing $x$.
Define
$$\eta(\Ffunc)(x) : \Ffunc(x) \to \Iffunc \circ \Jffunc (\Ffunc)(x)$$
to be the unique map from $\Ffunc(x)$ to the inverse limit $\Iffunc \circ \Jffunc (\Ffunc)(x)$ guaranteed by universal properties.
This definition of the unit further guarantees that whenever we have an elementary homotopy $\alpha \Rightarrow \beta$, the following diagram
commutes
\begin{equation*}
\xymatrix{
\Ffunc(x) \ar[rr]^{\Ffunc(\alpha = \beta)} \ar[d]_{\eta(\Ffunc)(x)}^\cong && \Ffunc(x')
\ar[d]^{\eta(\Ffunc)(x')}_\cong \\
\Iffunc \circ \Jffunc (\Ffunc)(x) \ar[rr]^{\Iffunc \circ \Jffunc (\Ffunc)(\alpha = \beta)} && \Iffunc \circ \Jffunc (\Ffunc)(x'),
}
\end{equation*}
thereby showing that

A morphism in $\psi : \Ffunc \to \Gfunc$ in $\big [ \Ent(X, \Sstrat), \Omega \big ]$
induces the following commutative diagram:
	\begin{equation*}
	\xymatrix{
	\Ffunc \ar[rr]^{\psi} \ar[d]^{\eta(\Ffunc)} && \Gfunc \ar[d]^{\eta(\Gfunc)} \\
	\Iffunc \circ \Jffunc (\Ffunc) \ar[rr]^{\Iffunc \circ \Jffunc(\psi)} && \Iffunc \circ \Jffunc (\Gfunc).
	}
	\end{equation*}
This completes the proof of our main equivalence theorem.
\end{proof}

\begin{cor}
\label{cor:local-system}
Suppose $(X,\Sstrat)$ is a conically stratified space with a unique stratum $X$.
Then the category of locally constant cosheaves $\Loc_{\Omega}(X)$ is equivalent to $\big[ \pi_1(X),\Omega \big]$.
\end{cor}

\begin{ex}\label{ex:zig-zag-modules}
Suppose $X=\R$ is stratified into $n$ vertices and $n+1$ open intervals.
If $\cosheaf{F}:\Open(\R) \to \Omega$ is a constructible cosheaf, then Theorem~\ref{thm:classification} implies that $\cosheaf{F}$ is equivalent to a functor $\tilde \Ffunc:\Ent(\R,\Sstrat) \to \Omega$, which by Example~\ref{ex:R} is equivalent to a functor modeled on the associated zig-zag category $Z$:
\[
\tilde \Ffunc|_Z: \qquad \tilde \Ffunc(y_0) \rightarrow \tilde \Ffunc(x_1) \leftarrow \tilde \Ffunc(y_1) \rightarrow \cdots \leftarrow \tilde \Ffunc(y_{n-1}) \rightarrow \tilde \Ffunc(x_n) \leftarrow \tilde \Ffunc(y_n).
\]
If $\Omega=\Vect_{\Bbbk}$, the functor $\tilde \Ffunc |_Z$ defines a \define{zigzag module}, see~\cite{carlsson2010zigzag}.
Our classification theorem implies that the category of $(\R,\Sstrat)$-constructible cosheaves valued in $\Vect_{\Bbbk}$ are equivalent to the category of zigzag modules.
\end{ex}

\begin{rmk}
We continue Example~\ref{ex:zig-zag-modules} to obtain a finer classification of $(\R,\Sstrat)$-constructible cosheaves valued in $\Vect_{\Bbbk}$.
We refer to any connected full subcategory $I$ of
the zig-zag category $Z$ as an \define{interval category}.
Here, connected means that any two objects in $I$ can be reached via a sequence of morphisms in $I$.
A functor $\Iffunc_{\Bbbk} : Z \to\Vect_{\Bbbk}$ that is constant with value $\Bbbk$ on $I$ and $0$ outside of $I$ is called an \define{interval module}.
A theorem of Gabriel~\cite{gabriel} implies that every functor $\Ffunc : Z \to \Vect_{\Bbbk}$ can be expressed uniquely as a direct sum of interval modules, assuming $\Bbbk$ is algebraically closed.
We note that since each object in the interval category $I$ determines a stratum in $(\R,\Sstrat)$, we can associate to $I$ a topological interval in $\R$ by taking the union of the strata specified by $I$.
For instance, if $(\R,\Sstrat)$ has at least five strata, then the interval category $x_1 \leftarrow y_1$ determines the topological interval $[x_1,x_2)$.
Theorem~\ref{thm:classification}, Example~\ref{ex:R}, and Gabriel's Theorem together imply that every $(\R,\Sstrat)$-constructible cosheaf valued in $\Vect_{\Bbbk}$ can be expressed uniquely as a direct sum of indecomposable cosheaves supported on intervals of four types: open, closed, half-open on the left, and half-open on the right.
%The corresponding statement for constructible sheaves holds as well.
\end{rmk}

\begin{figure}[h]
\centering
\includegraphics[width=\textwidth]{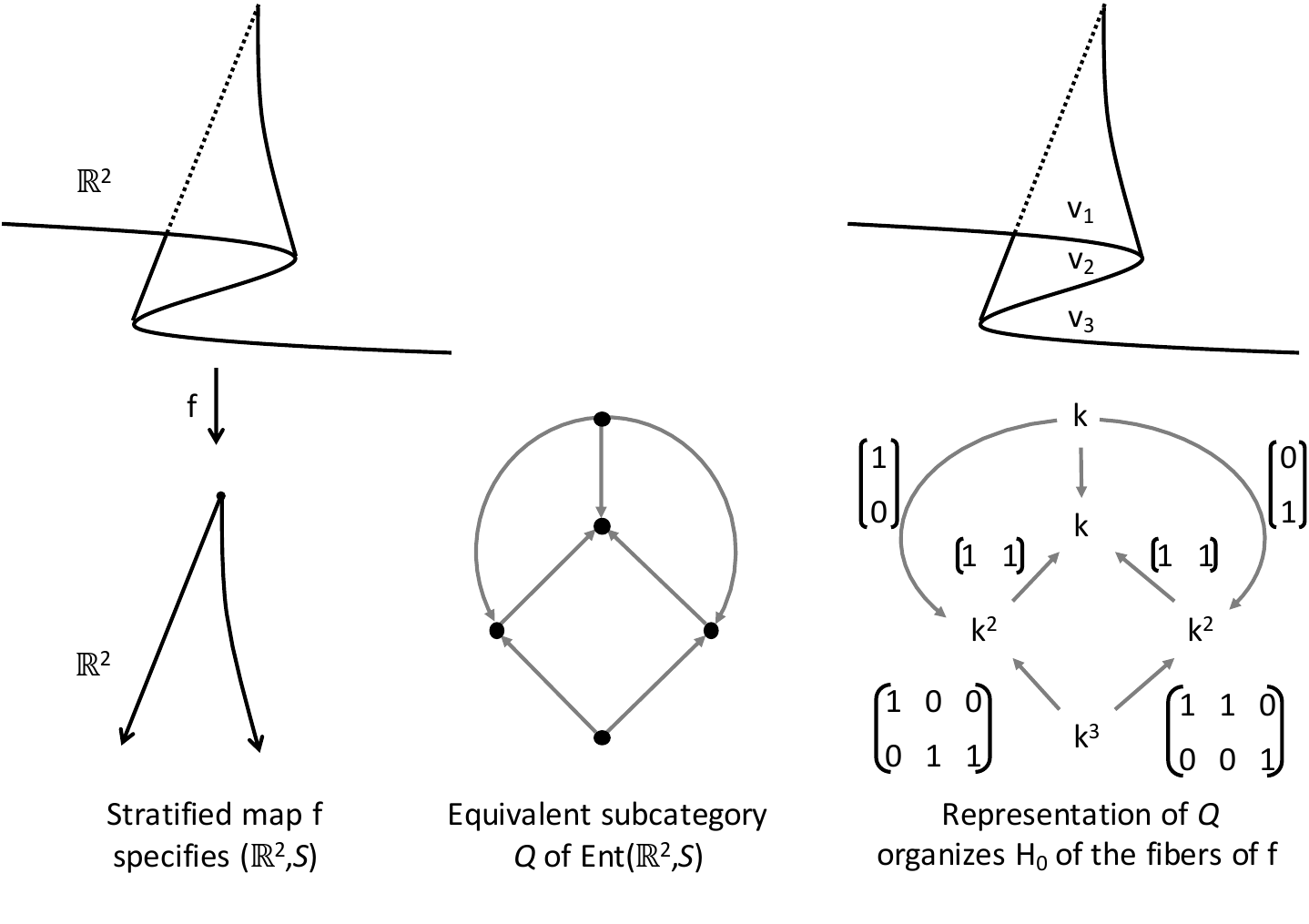}
\caption{Studying a map with a cusp and two folds using the Entrance Path Category}
\label{fig:cusp-example}
\end{figure}

\begin{ex}\label{ex:whitney-cusp}
Consider the Whitney cusp
\[
	f:\R^2 \to \R^2 \qquad f(x,y) := (xy-x^3,y).
\]
The importance of such a map was illustrated by Whitney~\cite{Whitney55}, who showed that cusps and folds are the only singularities that appear for structurally stable maps from the plane to the plane.
The cusp singularity introduces a zero-dimensional stratum and the two fold singularities introduce two one-dimensional strata in the codomain of $f$.
The complement of these strata define the two two-dimensional strata; see the left part of Figure~\ref{fig:cusp-example}.
Let $\cosheaf{F} : \Ent(\Rspace^2, \Sstrat) \to \Vect$ be the functor
that assigns to each point $x \in \Rspace^2$, the homology group $H_0 \big( f^{-1}(x) \big)$, i.e. $\Ffunc(x):= H_0 \big( f^{-1}(x) \big)$.
Given an entrance path $\alpha$ in $(\Rspace^2, \Sstrat)$ and a point $y \in f^{-1}\big( \alpha(0) \big)$,
lift $\alpha$ along $f$ to a path starting at $y$.
If two entrance paths $\alpha$ and $\beta$ are equivalent in $(\Rspace^2, \Sstrat)$ and a point
$y \in f^{-1} \big( \alpha(0) = \beta(0) \big)$ is given, then both lift to homotopic paths
in the domain of $f$ starting at $y$.
Thus an entrance path $\alpha$ in $(\Rspace^2, \Sstrat)$ induces a well defined map
$\cosheaf{F}(\alpha) : \cosheaf{F}\big( \alpha(0) \big) \to \cosheaf{F} \big( \alpha(1) \big)$.
By picking a representative point from each of the five strata described above, we obtain a subcategory $Q$ of $\Ent(\R^2,\Sstrat)$ with five objects.
Since each stratum is simply connected one can see that $Q$ is equivalent to $\Ent(\R^2,\Sstrat)$.
This greatly simplifies the description of the cosheaf $\cosheaf{F}$.
\end{ex}

%%%%%%%%%%%%%%%%%%%%%%%%%%%%%%%%%%%%%%%
\section{Stratified Coverings}
\label{sec:coverings}
%%%%%%%%%%%%%%%%%%%%%%%%%%%%%%%%%%%%%%%

Now we describe how a constructible cosheaf valued in $\Omega=\Set$ is equivalent to a stratified covering.
We fix an $n$-dimensional conically stratified space $(X,\Sstrat)$.
Recall that by virtue of the axiom of the frontier, the set of strata $\Sstrat$ is a poset where $S' < S$ if $S'\subseteq \overline{S}$ and $S\neq S'$.

\begin{defn}\label{defn:stratified-covering}
A stratum-preserving map $f : (Y, \Tstrat) \to (X, \Sstrat)$ of conically stratified spaces is a \define{stratified covering} of $(X, \Sstrat)$ if the following holds:
	\begin{itemize}
	\item For each stratum $S \in \Sstrat$ and for each stratum $T\subseteq f^{-1}(S)$ the restriction $f|_T : T \to S$ of $f$ is either a map from the empty set or a covering space.
	Recall that a \define{covering space} is a continuous map $p:\tilde{S} \to S$ for which every point $x \in S$ has a neighborhood $V$ so that the preimage $p^{-1}(V)$ is the disjoint union of open sets, each of which are mapped homeomorphically onto $V$ by $p$. In particular, a covering space is surjective.
	\item If $f$ takes $T \in \Tstrat$ to $S \in \Sstrat$ and there is an $S' \in \Sstrat$ such that $S' < S$,
	then there is a $T' \in \Tstrat$ such that $T' < T$ and $f$ takes $T'$ to $S'$.
	\end{itemize}
A \define{morphism} from a stratified covering $f : (Y, \Tstrat) \to (X, \Sstrat)$ to a stratified covering $f' : (Y', \Tstrat') \to (X, \Sstrat)$
is a stratum preserving map $\mu : (Y, \Tstrat) \to (Y', \Tstrat')$ such that $f = f' \circ \mu$.
Let $\Cov(X, \Sstrat)$ be the category of stratified coverings of $(X, \Sstrat)$.
\end{defn}

\begin{ex}[Reeb Graphs]\label{ex:reeb-graph}
Let $f:M \to \R$ be a Morse function on a compact manifold $M$.
Consider the equivalence relation on $M$ that identifies $x,y\in M$ if and only if $f(x)=f(y)=v$ and if there exists a path $\gamma: [0,1] \to f^{-1}(v)$ with $\gamma(0)=x$ and $\gamma(1)=y$.
Denote the quotient space of $M$ by this equivalence relation by $\bar{M}$. Note that since $f$ is constant on equivalence classes it factors to define a map $\bar{f}:\bar{M} \to \R$.
We claim that there exists a natural pair of stratified spaces $(\bar{M}, \Sstrat)$ and $(\R,\Tstrat)$ so that $\bar{f}:(\bar{M}, \Sstrat)\to (\R,\Tstrat)$ is a stratified covering.

It is classical result~\cite{reeb1946points,kronrod1950functions} that the quotient space $\bar{M}$ is a finite graph and that $\bar{f}$ restricts to a proper submersion over each edge; see~\cite[pp.390--391]{sharko2006kronrod} for a more recent treatment.
The map $\bar{f}:\bar{M} \to \R$ is called the \define{Kronrod-Reeb graph} construction associated to $f:M\to \R$.
To see that $\bar{f}$ is naturally a stratified covering, we first let $\Tstrat$ be the stratification of $\R$ into critical values
of $f$ as the zero-dimensional strata and the connected components of the complement as the one-dimensional strata.
Let $\Sstrat'$ be the natural stratification of $\bar{M}$ into vertices and open edges; which exists by virtue of the fact that $\bar{M}$ has the structure of a finite graph.
Let $\Sstrat$ be the coarsest refinement of $\Sstrat'$ where every point in the pre-image of a critical value is regarded as a zero-dimensional stratum.
It is easy to see that with these stratifications $\bar{f}:(\bar{M}, \Sstrat)\to (\R,\Tstrat)$ is a stratified covering.
\end{ex}

\begin{ex}\label{ex:z^n}
Suppose $X=\CC$ is stratified into two strata $S=\CC - \{0\}$ and $S'=\{0\}$.
The map $f(z)=z^n$ defines a stratified covering $f:(X,\Sstrat)\to (X,\Sstrat)$.
\end{ex}

\begin{ex}
The map considered in Example~\ref{ex:whitney-cusp} is also a stratified covering of the plane.
\end{ex}

The definition of a stratified covering implies the following technical lemma.

\begin{lem}
\label{lem:inverse-image}
Given a stratified covering $f : (Y, \Tstrat) \to (X, \Sstrat)$
and a basic open $U \in \Basic(X, \Sstrat)$, the pre-image $f^{-1}(U)$ is either empty
or a disjoint union of basic opens in $\Basic(Y, \Tstrat)$.
\end{lem}
\begin{proof}
Fix a point $x$ in a $d$-dimensional stratum $S$ and a basic open neighborhood
$U_x$ of $x$.

Suppose there is no $y$ such that $f(y)=x$.
The first property of a stratified covering
implies that $f^{-1}(U_x \cap S)$ is empty and hence there is no $T \in \Tstrat$ such that $f(T)=S$.
Consider a stratum $S' > S$ that intersects $U_x$.
The pre-image $f^{-1}(U_x \cap S')$ must also be empty,
because otherwise that would contradict the second property of a stratified covering.
This implies that $f^{-1}(U_x)$ is empty.

Suppose there is a $y\in Y$ such that $f(y)=x$.
Recall that we have a filtration-preserving homeomorphism
$h_x:\RR^d\times C(L_x)\to U_x$.
Consider a point of the form $x' = h_x(0,p,1)$.
This point has a path passing through it by letting
$\gamma_{x'}(t)= h_x(0,p,t)$ range over all values $t>0$.
This path is contained in a unique stratum $S'$.
The first property of a stratified covering implies that this path either has empty pre-image or
a set of lifts $\bar{\gamma}^{y'}_{x'}(t)$ indexed by the fiber $y'\in f^{-1}(x')$.
By continuity and the second property of a stratified covering,
the limit of $\bar{\gamma}^{y'}_{x'}(t)$ as $t$ tends to $0$ exists.
Let
$$C(\bar{L}^y_x) \quad = \quad \bigcup_{x' \in h_x \big( 0, (L_x,1) \big)} \Big \{ \bar{\gamma}^{y'}_{x'}(t) \; \Big | \; \lim_{t\to 0} \bar{\gamma}^{y'}_{x'}(t)=y \Big \} \cup \{y\}.$$
Note that the filtration of $h_x \big( 0, (L_x,1) \big) \cong L_x $ defines a filtration of $C(\bar{L}^y_x)$ as well.
By considering similar paths passing through each point of the form $x''=h_x(v,p,1)$ where $v$ ranges over $\RR^d$ we see that
$$f^{-1}(U_x) \quad \cong \quad \bigcup_{y\in f^{-1}(x)} \RR^d \times C(\bar{L}_x^y)$$
is a disjoint union of basic opens in $Y$.
\end{proof}

We now show that the category of stratified coverings and morphisms thereof is equivalent to the category of representations of the entrance path category valued in $\Set$.

\begin{thm}
\label{thm:stratified_covering}
$\Cov(X, \Sstrat)$ is equivalent to $\big [ \Ent(X, \Sstrat), \Set \big ]$.
\end{thm}
\begin{proof}
We construct two functors
$$\Iffunc : \Cov(X, \Sstrat) \to \big [ \Ent(X, \Sstrat), \Set \big ]$$
$$\Jffunc : \big [ \Ent(X, \Sstrat), \Set \big ] \to \Cov(X, \Sstrat)$$
such that
$\Jffunc \circ \Iffunc$ is naturally isomorphic to the identity functor on $\Cov(X, \Sstrat)$
and $\Iffunc \circ \Jffunc$ is naturally isomorphic to the identity functor on
$\big [ \Ent(X, \Sstrat), \Set \big ]$.

\paragraph{Construction of $\Iffunc$.}
Fix a stratified covering $f: (Y, \Tstrat) \to (X, \Sstrat)$.
For each point $x \in X$ we let
$\Iffunc(f)(x) := f^{-1}(x)$.
For an entrance path $\alpha : x' \to x$ in $(X, \Sstrat)$ that starts at $x'$ and ends at $x$ we argue that
each point $y' \in f^{-1}(x')$ specifies a lift of $\alpha$ to an entrance path in
$(Y, \Tstrat)$ starting at $y'$ and ending at a point in $f^{-1}(x)$, assuming $f^{-1}(x')\neq \varnothing$.
To see this, we first observe that if an entrance path is contained in a unique stratum in $(X,\Sstrat)$, then every point $y' \in f^{-1}(x')$ specifies a unique lift by virtue of the unique path-lifting property of covering spaces~\cite[pp.~66-68]{spanier1966algebraic}.
A general entrance path $\alpha$ that intersects multiple strata is equivalent to a composition of entrance paths
$\alpha_n\ast\cdots\ast\alpha_1$
where each path $\alpha_i$ sends $[0,1)$ to a unique stratum in
$(X,\Sstrat)$ and enters a different stratum at $t=1$.
This allows us to define the lift of $\alpha$ inductively.
We can choose a unique lift of $\alpha_i|_{[0,1)}$ for each $y_i \in f^{-1}( \alpha_i (0))$ because it is contained in a unique stratum.
Lemma~\ref{lem:inverse-image} implies that any basic open $U_i$ about $\alpha_i(1)$ has a pre-image that is homeomorphic to a disjoint union of basic opens, one about each point $y_{i+1}\in f^{-1}(\alpha_i(1))=f^{-1}(\alpha_{i+1}(0))$.
We say that a lift of $\alpha_i|_{[0,1)}$ \define{specializes} to a point $y_{i+1} \in f^{-1}(\alpha_i(0))$ if the basic open about $y_{i+1}$ intersects the lift of $\alpha_i|_{[0,1)}$.
The specialization of a lift of $\alpha_i|_{[0,1)}$ is also the continuous limit of that lift.
Repeating this procedure for $\alpha_{i+1}|_{[0,1)}$ allows us to inductively lift a general entrance path $\alpha$.
Consequently, an entrance path $\alpha$ induces a map of sets
$$\Iffunc(f)(\alpha : x' \to x) : \Iffunc(f)(x') \to \Iffunc(f)(x).$$
Moreover, the above reasoning shows that an elementary homotopy $\alpha \Rightarrow \beta$ lifts to a collection of elementary homotopies
in $(Y, \Tstrat)$ thereby making the map
$\Iffunc(f)(\alpha) = \Iffunc(f)(\beta)$.

\paragraph{Construction of $\Jffunc$.}
Given a representation $\cosheaf{F} : \Ent(X, \Sstrat) \to \Set$,
we construct a stratified covering $f : (Y, \Tstrat) \to (X, \Sstrat)$.
As a set, let $Y := \bigsqcup_{x \in X} \cosheaf{F}(x)$ and let $f : Y \to X$ be the natural projection map.
Now we topologize $Y$.
Recall from Proposition~\ref{prop:terminal-object} that for any $x \in X$ and any basic open neighborhood $U_x$ of $x$, any two entrance paths
in $U_x$ with a common starting point $x'$ that end at $x$ are equivalent.
Consequently, for each point $x \in X$ and for each basic open neighborhood $U_x$ of $x$, and each $y \in \cosheaf{F}(x)$, let
\[
	\bar{U}_x^y := \bigcup_{x'\in U_x} \Big \{ y' \in \Ffunc(x') \;  \Big | \; \cosheaf{F}(\alpha : x' \to x)(y') = y \Big \}.
\]
Here $\alpha_x : x' \to x$ is any entrance path in $U_x$ starting at $x'$ and ending at $x$.
The set $\big \{ \bar{U}_x^y \big \}$ over all $x \in X$, $U_x$, and $y \in \cosheaf{F}(x)$ forms
the basis for the topology on $Y$.
The projection $f: Y \to X$ is now a continuous map.
We now put on $Y$ the structure of a conically stratified space $(Y, \Tstrat)$ so that
$f : (Y, \Tstrat) \to (X, \Sstrat)$ is a stratified covering.
By setting $Y^d= f^{-1}(X^d)$,
the $n$-step filtration of $X$ defines an $n$-step filtration of $Y$.
We check that that this filtration satisfies the properties of Definition~\ref{defn:stratified-space}.
Consider a point $y\in Y^d - Y^{d-1}$ and let
\[
\bar{L}_x^y = \bigcup_{x' \in h_x \big(0,(L_x,1) \big)} \Big \{ y' \in \cosheaf{F}(x')
\; \Big | \; \cosheaf{F}(\alpha: x' \to x)(y') = y  \Big \}
\]
where $h_x:\RR^d\times C(L_x)\to U_x$ is the filtration-preserving homeomorphism that coordinatizes $U_x$.
Note that $\bar{L}_x^y$ has an $(n-d-1)$-step filtration because $L_x$ does.
If $S$ is the stratum containing $x=f(y)$, then
$\cosheaf{F}(\alpha_{x'}^x)$ is an isomorphism for each $x' \in S \cap U_x$.
This means that $f$ restricts to a homeomorphism
$\bar{U}_x^y \cap f^{-1}(S) \to S \cap U_x \cong \Rspace^d$ making
$f$ a covering space over each stratum in $\Sstrat$ with fiber $\cosheaf{F}(x)$.
Finally, observe that if $x'=h_x(v,p,1)$ is in a stratum $S'$, then any point $x''=h_x(v,p,t)$ is also
in $S'$ for all $t > 0$.
We have
\[
\bar{U}_x^y \cong \RR^d \times C(\bar{L}_x^y),
\]
thereby proving that $Y$ is a stratified space.
To check the second property of a stratified covering,
suppose that $S'< S$.
Let $\alpha$ be any entrance path from a point $x \in S$ to a point $x' \in S'$.
If $f^{-1}(x):=\cosheaf{F}(x) \neq \emptyset$, then $\cosheaf{F}(x')=:f^{-1}(x') \neq \emptyset$ because the set map
$\cosheaf{F}(\alpha):F(x) \to F(x')$ cannot be a map to the empty set, because in the category $\Set$ there are no morphisms to the empty set other than the identity morphism.

\paragraph{Construction of $\epsilon$.}
We now prove that the co-unit $\epsilon: \Jffunc \circ \Iffunc \Rightarrow \id_{\Cov}$ is a natural isomorphism.
The co-unit is defined as follows: given a stratified cover $f: (Y, \Tstrat) \to (X, \Sstrat)$ the associated stratified cover $\Jffunc \circ \Iffunc(f)$ is defined by $\bigsqcup_{x \in X} \Iffunc(f)(x)$, where $\Iffunc(f)(x)=f^{-1}(x)$.
Since for every $x\in X$, we have a map $f^{-1}(x) \hookrightarrow Y$ given by inclusion, the universal property of the coproduct provides a universal map $\Jffunc \circ \Iffunc(f)=\bigsqcup_{x \in X} \Iffunc(f)(x) \to Y$.
We set $\epsilon(f):\Jffunc \circ \Iffunc(f) \to Y$ to be this universal map.
Note that by construction, the basic opens of $\Jffunc \circ \Iffunc(f)$ are in bijection with basic opens of
$(Y, \Tstrat)$ so $\epsilon(f)$ is a stratum preserving homeomorphism.

To check naturality of $\epsilon$ consider a morphism from a stratified covering space $f : (Y, \Tstrat) \to (X, \Sstrat)$ to
$f' : (Y', \Tstrat') \to (X, \Sstrat)$ given by a stratum preserving map
$\mu : (Y, \Tstrat) \to (Y', \Tstrat')$ where $f = f' \circ \mu$.
Since $\mu$ must commute with the map to $(X,\Sstrat)$, we have that for every point $x\in X$, $\mu_x:f^{-1}(x) \to f'^{-1}(x)$.
This implies that we have an associated map between the disjoint unions
\[
	\epsilon(\mu): \bigsqcup_{x \in X} \Iffunc(f)(x) \to \bigsqcup_{x \in X} \Iffunc(f')(x).
\]
It is easy to check that this map is a stratum-preserving map assuming that $\mu$ is.

\paragraph{Construction of $\eta$.}
We now prove that the unit $\eta: \id_{[ \Ent(X, \Sstrat), \Set ]} \Rightarrow \Iffunc \circ \Jffunc$ is a natural isomorphism, and is in fact the identity.
Recall that $\Jffunc$ takes a representation of the entrance path category $\cosheaf{F} : \Ent(X, \Sstrat) \to \Set$ and turns it into a stratified cover, where the fiber over a point $x\in X$ is $\cosheaf{F}(x)$.
Applying $\Iffunc$ to construct a representation of the entrance path category valued in $\Set$ then assigns the same set $\cosheaf{F}(x)$ to $\Iffunc \circ \Jffunc (\cosheaf{F})(x)$.
Furthermore, an elementary homotopy $\alpha \Rightarrow \beta$ between entrance paths in
$\Ent(X, \Sstrat)$ lifts to a collection of elementary homotopies in $\Jffunc (\cosheaf{F})$.
This means that if $\gamma$ and $\gamma'$ are arbitrary and equivalent entrance paths in $(X, \Sstrat)$,
then $\Iffunc \circ \Jffunc(\cosheaf{F})$ send both $\gamma$ and $\gamma'$ to $\Ffunc(\gamma) = \Ffunc(\gamma')$.
Therefore $\Iffunc \circ \Jffunc(\cosheaf{F}) = \cosheaf{F}$.
For any natural transformation $\phi : \Ffunc \to \Gfunc$, $\Ifunc \circ \Jfunc (\phi)$ is equal to the map
$\phi(x) : \Ffunc(x) \to \Gfunc(x)$ for all $x \in X$, and therefore $\Ifunc \circ \Jfunc(\phi) = \phi$.
This completes the proof.
\end{proof}

The following Corollary is an immediate consequence of
Theorem \ref{thm:classification} and Theorem \ref{thm:stratified_covering}.

\begin{cor}\label{cor:strat-cover-ent-path}
$\Cov(X, \Sstrat)$ is equivalent to $\Csh_{\Set}(X, \Sstrat)$.
\end{cor}

\begin{rmk}
When $\Sstrat$ consists of a single stratum, namely all of $X$,
Theorem \ref{thm:stratified_covering}
restricts to the classical classification of covering spaces over a manifold.
\end{rmk}

\begin{rmk}
Theorem~\ref{thm:stratified_covering} and Corollary~\ref{cor:strat-cover-ent-path} together generalize the
categorification of Reeb graphs developed in~\cite{CRG}, which was concerned with the case $X=\R$,
to Reeb spaces \cite{Reeb_spaces}.
\end{rmk}

We conclude this section with two examples and a question.

\begin{ex}\label{ex:C-finite}
Suppose $X=\CC$ is stratified into two strata $S=\CC - \{0\}$ and $S'=\{0\}$.
Let $x \in S$.
In view of Example~\ref{ex:C},
a functor $\cosheaf{F}:\Ent(X,\Sstrat)\to\Set$ is determined by
two sets
$\cosheaf{F}(x)$ and $\cosheaf{F}(0)$,
a permutation $\cosheaf{F}(\sigma) : \cosheaf{F}(x) \to \cosheaf{F}(x)$,
and a map $\cosheaf{F}(\alpha) : \cosheaf{F}(x) \to \cosheaf{F}(0)$,
where $\cosheaf{F} (\alpha) \circ \cosheaf{F}(\sigma) = \cosheaf{F}(\alpha)$.
The sets $\cosheaf{F}(x)$ and $\cosheaf{F}(0)$
define the fibers $f^{-1}(x)$ and $f^{-1}(0)$
of a stratified cover.
The morphism $\cosheaf{F}(\alpha)$ defines the \define{specialization map}
from $f^{-1}(x)$ to $f^{-1}(0)$.
The condition that $\cosheaf{F} (\alpha) \circ \cosheaf{F}(\sigma) = \cosheaf{F}(\alpha)$ implies that
if two points in the fiber $f^{-1}(x)$ are related by a deck transformation,
then they must specialize to the same element of $f^{-1}(0)$.
In particular, the only connected stratified covers of $(X,\Sstrat)$ with finite fibers are those of the form $z^n$.
\end{ex}

The following example was suggested by an anonymous referee, which we have decided to include.

\begin{ex}\label{ex:C-infinite}
As a continuation of Example~\ref{ex:C-finite}, consider the representation of the entrance path category where any point $x\neq 0$ has a countably infinite set $\cosheaf{F}(x)\cong \Z$ assigned to it and $\cosheaf{F}(0)=\{\star\}$ is the one point set.
Suppose further that the deck transformation $\cosheaf{F}(\sigma):\Z \to \Z$ is the map that sends $n \mapsto n+1$.
Necessarily, $\cosheaf{F}(\alpha):\Z \to \{\star\}$ is the constant map.
If we consider the stratified cover $\Jffunc(\cosheaf{F})$ associated to $\cosheaf{F}$ in the proof of Theorem~\ref{thm:stratified_covering}, then we see that any basic open around $0$ in $\CC$ has a pre-image that is a basic open that is homeomorphic to the cone on the real line.
Moreover, this stratified cover---when restricted to the link--- restricts to the universal covering of the circle.
\end{ex}

Finally, we note that studying the automorphism group of a stratified covering is a very interesting problem that must generalize the usual Galois correspondence for covering spaces.
We leave this for future work.

\section{Conclusion}

The idea that there are nice combinatorial models for certain sheaves and cosheaves is not new, but it is one that continues to generate useful insights.
One of the first accounts in this vein was Zeeman's 1954 thesis~\cite[pp.~626-627]{zeeman1962dihomology}, which developed a theory of local coefficient systems that were modeled on the face-relation poset of a cell complex.
Zeeman made allusions to Leray's sheaf theory, but the connection with the sheaf axiom was not made explicit.
Although some of Zeeman's contributions were forgotten, they foreshadowed later results by Kashiwara~\cite{kashiwara1984riemann} and, separately,
Shepard~\cite{shepard}, who was a PhD student of MacPherson's.
Both Kashiwara and Shepard were concerned with a full treatment of the derived category of sheaves that are constructible with respect to a simplicial or cellular stratification and thus focused on sheaves valued in abelian categories.

At some point between 1984 and 2007, a parsimonious description of the category of $\Sstrat$-constructible sheaves, for stratifications $(X,\Sstrat)$ more general than simplicial or cellular, must have emerged.
Indeed, in Treumann's 2007 thesis, he outlined an unpublished result of MacPherson's---that $\Sstrat$-constructible sheaves valued in $\Set$ were equivalent to functors from the exit path category into $\Set$.
Treumann's thesis treated a 2-categorical analog of this result~\cite{treumann} and by 2009, Jacob Lurie had already sketched an $\infty$-category version of MacPherson's result in ``Derived Algebraic Geometry: Volume VI''~\cite{lurieDAG6}.
Since then, higher-categorical treatments of MacPherson's insight have gained attention from multiple authors---we refer to the introductory section of~\cite{barwick2018exodromy} on ``exodromy for topological spaces'' for a more complete bibliography of this thread of development.

Despite all of the above work, this paper provides the first self-contained treatment of MacPherson's result in a 1-categorical setting, adapted to constructible \emph{co}sheaves valued in a bi-complete category.
Our motivation to provide such a treatment comes from the increased use of sheaves and cosheaves in applications outside of mathematics; see~\cite{curry,CRG,munch2016convergence,kashiwara2018persistent,macpherson2018persistent,hansen2019toward} for some notable examples.
We believe that having a presentation that uses minimal technical overhead will lead other researchers to more creative uses of this theory.
Additionally, by rephrasing the constructibility hypothesis, our work suggests that an even more immediate proof of MacPherson's equivalence should be available: By localizing the poset of open sets along the ``associated to the same stratum'' relation, one should be able to produce the entrance path category directly.
Such a proof, if possible, would provide a conceptually minimal account of an observation that has inspired significant mathematics over the course of many years.

\section*{Acknowledgements}

The authors would like to thank Robert MacPherson for his support and mentorship over the years. 
The authors would also like to thank an anonymous referee for their thorough reading and their corrections to mistakes found in earlier drafts of this paper.
The paper is much stronger as a result of their efforts.

\bibliographystyle{plain}

\end{document}